\documentclass{article}
\usepackage{amsmath}
\usepackage{amsthm}
\usepackage{times}
\usepackage{bm}
\usepackage[font=small,skip=0pt]{caption}
\usepackage{subcaption}
\usepackage{graphicx}
\usepackage{amssymb}
\usepackage{pifont}
\usepackage[page,title]{appendix}
\usepackage{mathrsfs}
\usepackage[margin=1.2in]{geometry}
\usepackage{setspace}
\usepackage{fancyhdr} 
\usepackage{listings}
\usepackage{xcolor}
\usepackage{natbib}
\usepackage{booktabs}
\usepackage{multirow}
\usepackage{enumerate}
\usepackage{appendix}
\usepackage{authblk}
\usepackage[ruled,vlined]{algorithm2e}

\usepackage{thmtools}
\usepackage{thm-restate}

\allowdisplaybreaks[4]
\newcommand{\WFcomment}[1]{{\color{red!80!black} {\textbf [#1]}}}

\newcommand{\simiid}{\stackrel{\text{i.i.d.}}{\sim}}
\newcommand{\simind}{\stackrel{\text{ind.}}{\sim}}
\newcommand{\EE}{\mathbb{E}}
\newcommand{\PP}{\mathbb{P}}
\newcommand{\RR}{\mathbb{R}}
\newcommand{\diag}{\text{diag}}
\DeclareMathOperator{\logit}{logit}
\DeclareMathOperator{\supp}{supp}
\DeclareMathOperator{\sgn}{sgn}
\DeclareMathOperator{\Cov}{Cov}
\DeclareMathOperator{\Var}{Var}
\DeclareMathOperator*{\argmax}{arg\,max}
\newcommand{\TPR}{\textnormal{TPR}}
\newcommand{\setcomp}{\mathsf{c}}
\newcommand{\toProb}{\stackrel{p}{\to}}
\newcommand*{\tran}{{\mkern-1.5mu\mathsf{T}}}

\newcommand{\ep}{\varepsilon}

\newcommand{\tX}{\widetilde{X}}
\newcommand{\tbeta}{\widetilde{\beta}}

\newcommand{\tsigma}{\widetilde{\sigma}}

\newcommand{\tp}{\widetilde{p}}

\newcommand{\hbeta}{\widehat{\beta}}
\newcommand{\htheta}{\widehat{\theta}}
\newcommand{\hrho}{\widehat{\rho}}
\newcommand{\hsigma}{\widehat{\sigma}}
\newcommand{\hSigma}{\widehat{\Sigma}}
\newcommand{\hDelta}{\widehat{\Delta}}
\newcommand{\hk}{\hat{k}}

\newcommand{\hM}{\widehat{M}}
\newcommand{\hV}{\widehat{V}}

\newcommand{\XAug}{[X \; \tX]}

\newcommand{\cN}{\mathcal{N}}
\newcommand{\cP}{\mathcal{P}}
\newcommand{\cR}{\mathcal{R}}

\newcommand{\TPP}{\textnormal{TPP}}

\newcommand{\FDP}{\textnormal{FDP}}
\newcommand{\hFDP}{\widehat{\textnormal{FDP}}}

\newcommand{\FWER}{\textnormal{FWER}}
\newcommand{\kn}{\textnormal{kn}}
\newcommand{\wh}{\textnormal{wh}}

\newtheorem{theorem}{Theorem}
\newtheorem{proposition}{Proposition}

\newtheorem{cor}{Corollary}
\newtheorem{lemma}{Lemma}

\newtheorem{example}{Example}

\title{Whiteout: when do fixed-$X$ knockoffs fail?}
\author{Xiao Li and William Fithian\\
Department of Statistics, UC Berkeley}
\date{\today}

\begin{document}

\maketitle

\abstract{
A core strength of knockoff methods is their virtually limitless customizability, allowing an analyst to exploit machine learning algorithms and domain knowledge without threatening the method's robust finite-sample false discovery rate control guarantee. While several previous works have investigated regimes where specific implementations of knockoffs are provably powerful, general negative results are more difficult to obtain for such a flexible method. In this work we recast the fixed-$X$ knockoff filter for the Gaussian linear model as a conditional post-selection inference method. It adds user-generated Gaussian noise to the ordinary least squares estimator $\hbeta$ to obtain a ``whitened'' estimator $\tbeta$ with uncorrelated entries, and performs inference using $\sgn(\tbeta_j)$ as the test statistic for $H_j:\; \beta_j = 0$. We prove equivalence between our whitening formulation and the more standard formulation involving negative control predictor variables, showing how the fixed-$X$ knockoffs framework can be used for multiple testing on any problem with (asymptotically) multivariate Gaussian parameter estimates. Relying on this perspective, we obtain the first negative results that universally upper-bound the power of all fixed-$X$ knockoff methods, without regard to choices made by the analyst. Our results show roughly that, if the leading eigenvalues of $\text{Var}(\hbeta)$ are large with dense leading eigenvectors, then there is no way to whiten $\hbeta$ without irreparably erasing nearly all of the signal, rendering $\sgn(\tbeta_j)$ too uninformative for accurate inference. We give conditions under which the true positive rate (TPR) for any fixed-$X$ knockoff method must converge to zero even while the TPR of Bonferroni-corrected multiple testing tends to one, and we explore several examples illustrating this phenomenon.
}

\section{Introduction}

\subsection{The knockoff filter in the Gaussian linear model}\label{sec:knockoffsfirst} 

Knockoff methods are a flexible class of multiple testing procedures that operate by introducing a ``negative control'' or ``knockoff'' for each predictor variable in a supervised learning problem and then testing a learning algorithm's ability to distinguish each real predictor from its knockoff. The original method, the {\em fixed-$X$ knockoff filter} \citep{barber2015controlling}, is a multiple testing method for the Gaussian linear model where we observe design matrix $X\in\RR^{n\times d}$ and response 
\[y = X\beta + \ep, \quad\text{ where }  \ep_i \simiid \cN(0,\sigma^2), \quad i = 1,\ldots, n.\]
The parameters $\sigma^2 > 0$ and $\beta\in\RR^d$ are unknown, and the goal is to test the hypothesis $H_j:\; \beta_j = 0$ against the two-sided alternative, for $j = 1,\ldots,d$. Following \citet{barber2015controlling}, we assume throughout that $X$ has full column rank with $2d \leq n$.

For testing $H_j$ individually, the usual two-sided $t$-test is uniformly most powerful unbiased. That test rejects for extreme values of the $t$-statistic $\hbeta_j/\sqrt{\hsigma^2 \Sigma_{jj}}$, where $\Sigma = (X^\tran X)^{-1}$ and $\hbeta$ and $\hsigma^2$ are respectively the OLS estimator
\[
\hbeta \;=\; (X^\tran X)^{-1}X^\tran y \;\sim\; \cN_d\left(\beta,\, \sigma^2\Sigma\right)
\]
and the residual variance $\hsigma^2 = \|y - X\hbeta\|^2/(n-d)$. Taken together, these two estimators are a complete sufficient statistic for the model. Let $p_j$ denote the $p$-value for the two-sided $t$-test on $H_j$.

For multiple testing, a classical approach would reject $H_j$ when $p_j$ is sufficiently small, after making an appropriate correction for the multiplicity of tests. If $R$ is the number of rejections and $V$ is the number of true null hypotheses rejected (false discoveries), \cite{benjamini1995controlling} define the {\em false discovery proportion} as $\FDP = V/\max\{R,1\}$, and the {\em false discovery rate} (FDR) as its expectation. While the Benjamini--Hochberg (BH) procedure of \cite{benjamini1995controlling} is not known to control the FDR in this problem unless the columns of $X$ are orthogonal, recent methods can directly adjust BH for the multivariate $t$ dependence, guaranteeing FDR control while performing similarly to BH \citep{fithian2020conditional}. The FDR criterion relaxes the more conservative {\em family-wise error rate} $\FWER = \PP(V \geq 1)$, the probability of making any false rejections, which we could control using the conservative Bonferroni correction that rejects $H_j$ when $p_j \leq \alpha/m$.\footnote{A more accurate FWER correction would apply the closure of the max-$t$ test: if $S$ is the maximal set for which $\|T_S\|_\infty$ is below its $1-\alpha$ quantile under $H_S:\; \beta_S = 0$, then we can reject $H_j$ for $j \notin S$; see \citet{marcus1976closed} and \citet{hothorn2008simultaneous} for more details.}

The knockoff filter of \citet{barber2015controlling} takes a radically different approach, bypassing the $t$-test $p$-values entirely. The method begins by augmenting the design matrix $X$ with a second matrix $\tX \in \RR^{n \times d}$ of negative control or {\em knockoff} variables, constructed to satisfy
    \[
    \tX^\tran \tX = X^\tran X, \quad\text{and}\quad \tX^\tran X = X^\tran X - D,
    \]
for some diagonal matrix $D \preceq 2X^\tran X$, where $\preceq$ denotes the positive semidefinite ordering. As we will discuss, a larger value of $D_{jj}$ preserves more signal for the inference on $H_j$, but in general it is not possible to maximize all $D_{jj}$ simultaneously so tradeoffs must be made between variables.

Knockoffs then calculates so-called $W$-statistics $W_1,\ldots, W_d$ satisfying two properties:
\begin{description}
    \item[1. Sufficiency] $W_j$ depends on $\XAug$ and $y$ only through $\XAug^\tran \XAug$ and $\XAug^\tran y$, and
    \item[2. Antisymmetry] Swapping any variable $X_j$ with its knockoff $\tX_j$ would flip the sign of $W_j$ and leave every other $W_k$ fixed. That is, if $\text{Swap}_j(\cdot)$ exchanges the $j$th and $(d+j)$th columns of an input matrix, then
        \begin{equation*}
        W_k(\text{Swap}_j(\XAug), y) \;\;=\;\; 
        \begin{cases} \;-W_j(\XAug, y) & \text{ if } k =j \\[10pt]
        \;W_k(\XAug, y) & \text{ if } k \neq j\end{cases}. 
        \end{equation*}
    \end{description}

The absolute values $|W| = (|W_1|, \ldots, |W_d|)$ determine a data-adaptive hypothesis ordering, with larger values assigned higher priority. Sufficiency and antisymmetry, along with the Gaussian modeling assumptions, ensure two key distributional properties: first, that $\sgn(W_1),\ldots,\sgn(W_d)$ are conditionally independent given $|W|$; and second, that if $\beta_j=0$ then $\sgn(W_j) \sim \text{Unif}\{-1,+1\}$ given $|W|$, unless $W_j = 0$. 

After calculating $W$-statistics, the knockoff filter applies an ordered multiple testing method called {\em Selective SeqStep} \citep{barber2015controlling} treating each $\sgn(W_j)$ as a ``binary $p$-value'' for $H_j$. To control FDR at level $\alpha$, the knockoff+\footnote{\citet{barber2015controlling} also propose a version without the ``$1+$'' in the numerator, which controls a relaxed FDR criterion. All results in this paper pertain to the knockoff+ method, but they also apply with very minor modifications to the more liberal version.} method rejects all hypotheses $H_j$ for which $W_j$ exceeds the adaptive threshold $\hat{t}$:
\[
\hat{t} = \min\left\{t: \hFDP_t^{\kn} \leq \alpha\right\}, \quad\text{where}\quad \hFDP_t^{\kn}  = \frac{1 + \#\{j:\;W_j \leq -t\}}{\#\{j:\; W_j \geq t\}}.
\]

The fixed-$X$ knockoff filter is a highly versatile method that offers its users a multitude of choices in selecting both the knockoff matrix and the $W$-statistics. \citet{spector2020powerful} show the importance of choosing $D$ well and discuss ramifications on the procedure's power, and various other works detail myriad ways to tailor the $W$-statistics using machine learning methods that exploit structural assumptions or other prior beliefs about the coefficients \citep[see e.g.,][]{dai2016knockoff,katsevich2019multilayer,chen2020prototype,dai2021kernel}. The method's customizability poses a major challenge if we hope to bound its power uniformly over the analyst's entire choice set, since well-informed analysts have ample opportunities to stack the deck in their own favor.

In light of this tremendous flexibility, it may come as a surprise to discover regimes where {\em no} knockoffs method --- not even one designed with full knowledge of the true regression coefficients --- can achieve nontrivial power, even while Bonferroni-corrected inference achieves near-perfect power. To begin to explain how knockoffs can go wrong, Section~\ref{sec:knockoffswhitened} formally recasts the knockoff filter as a conditional post-selection inference method built around a randomized estimator $\tbeta = \hbeta + \omega$, where $\omega$ is user-generated Gaussian noise in the style of \citet{tian2018selective}. Our interpretation builds on a conditioning argument in \citet{barber2019knockoff} and an observation in \citet{sarkar2021adjusting} that knockoffs constructs two independent estimators for $\beta$.

\subsection{Knockoffs as conditional inference on a whitened estimator}\label{sec:knockoffswhitened}

We now give an alternative but equivalent account of the knockoff filter without $W$-statistics, without sufficiency and antisymmetry properties, and even without knockoff variables. We view the method instead as a conditional post-selection inference procedure, where the key step is to construct a {\em whitened estimator} $\tbeta$ with diagonal covariance. If $\Delta \in \RR^{d \times d}$ is any diagonal matrix with $\Delta \succeq \Sigma = (X^\tran X)^{-1}$, let
\begin{equation}\label{eq:tbeta}
\tbeta \;=\; \hbeta + \omega \;\sim\; \cN_d\left(\beta, \,\sigma^2 \Delta\right), \quad \text{ where } \;\;\omega \sim \cN_d\left(\,0,\;\sigma^2 (\Delta - \Sigma)\right)
\end{equation}
is noise generated by the user, independently of $\hbeta$. We will see in Section~\ref{sec:knockoffswithoutknockoffs} that even when $\sigma^2$ is unknown, $\omega$ can be ``carved'' out of $\hsigma^2$ as long as $n \geq d + r$, where $r = \text{rank}(\Delta - \Sigma) \leq d$.

The independence of the coordinates of $\tbeta$ is bought at the price of higher variance, since $\Delta_{jj} \geq \Sigma_{jj}$. However, this price can be recouped by an exploratory analysis using the statistic
\begin{equation}\label{eq:xi}
\xi \;=\; \Sigma^{-1} \hbeta - \Delta^{-1}\tbeta \;\sim\; \cN_d\left( A\beta, \,\sigma^2A\right), \quad \text{ where } \;A \;=\; \Sigma^{-1} - \Delta^{-1} \succeq 0.
\end{equation}
$\xi$ carries the information lost by whitening since $\hbeta = \Sigma(\xi + \Delta^{-1}\tbeta)$, and it is independent of $\tbeta$ since
\[
\Cov(\xi, \tbeta) \;=\; \Sigma^{-1}\Cov(\hbeta, \hbeta + \omega) - \Delta^{-1} \Cov(\tbeta, \tbeta) \;=\; \sigma^2 I_d - \sigma^2 I_d \;=\; 0.
\]
Thus, $\hbeta$ is effectively ``split'' into two independent Gaussian vectors, $\tbeta$ and $\xi$, each holding part of the information about $\beta$. In terms of $\omega, \tbeta$ and $\xi$, the fixed-$X$ knockoff filter can be equivalently defined as follows:
\begin{description}
\item[Stage 1 (whitening).] For any diagonal $\Delta \succeq \Sigma$, generate noise $\omega \sim \cN_d(0,\sigma^2(\Delta - \Sigma))$ independently of $\hbeta$. Calculate $\tbeta$ and $\xi$ as in (\ref{eq:tbeta}--\ref{eq:xi}), but do not reveal any data to the analyst yet.
\item[Stage 2 (exploratory analysis).] Let the analyst observe $\xi$ and $|\tbeta|$ and use them to order the $d$ hypotheses for Selective SeqStep as $H_{[1]}, \ldots, H_{[d]}$, where $[1]$ indexes the first hypothesis in order and $[d]$ the last. In addition, select a one-sided alternative for each $H_j$. Let $\psi_j = +1$ if the right-tailed alternative is selected, and $\psi_j = -1$ for the left-tailed alternative.
\item[Stage 3 (confirmatory analysis).] Using Selective SeqStep, test the hypotheses in order using $\sgn(\tbeta_j)$ as a conditional test statistic for $H_j$. The signs are conditionally independent given $\xi$ and $|\tbeta|$, with 
\begin{equation}\label{eq:logit}
    \logit\,\PP\left(\sgn(\tbeta_j) = +1 \;\big|\; \xi, \,|\tbeta|\right) \;=\; \frac{2\,|\tbeta_j|}{\sigma^2\Delta_{jj}} \cdot \beta_j, \quad \text{ where } \logit\,p \,=\, \log \frac{p}{1-p}.
\end{equation}
\end{description}

Let $\tp_j$ denote the $p$-value for the conditional test of $H_j$ against a right- or left-tailed alternative according to $\psi_j$. Because $\sgn(\tbeta_j) \sim \text{Unif}\{-1,+1\}$ under $H_j$, and is stochastically increasing in $\beta_j$, we have $\tp_j = 1/2$
when $\sgn(\tbeta_j) = \psi_j$, and $\tp_j=1$ when $\sgn(\tbeta_j) = -\psi_j$. 

Using these $p$-values, the Selective SeqStep method  rejects $H_{[j]}$ if $j \leq \hk$ and $\tp_{[j]} = 1/2$, where
\[
\hk = \max \left\{k:\; \hFDP_k^{\wh} \leq \alpha\right\}, \quad \text{ with } \;\; \hFDP_k^{\wh} = \frac{1 + \sum_{j=1}^k 1\{\tp_{[j]} = 1\}}{\sum_{j=1}^k 1\{\tp_{[j]} = 1/2\}}.
\]

The conditional $p$-values $\tp_1,\ldots,\tp_d$ are independent given $\xi$ and $|\tbeta|$, satisfying Selective SeqStep's requirements, so the method described above controls the FDR, both conditionally and marginally. In Section~\ref{sec:equivalence} we formally state and prove equivalence between this formulation and the usual formulation of knockoffs, building on a conditioning argument in the Supplement of~\citet{barber2019knockoff}.

The exploratory analysis is defined vaguely because the analyst can use $\xi$ and $|\tbeta|$ however they like, provided they have not yet observed anything else about $y$ (while the sufficiency property of \citet{barber2015controlling} also restricts how the analyst may use the design matrix $X$, we will see in Section~\ref{sec:equivalence} that no such rule is needed). It is in this unrestricted stage that knockoffs can freely exploit prior information and structural assumptions. Indeed, the analyst can even use an informal process that iterates between trying out several models, visualizing the (observable) data, consulting their intuitions or their colleagues, and so on.

When the exploratory stage goes well, the ordering is highly informative, effectively reducing the multiplicity in Stage 3 by focusing inferential power on the first few hypotheses. In many problems, a good exploratory analysis can more than compensate for the randomization and binarization of the confirmatory test statistics $\sgn(\tbeta_j)$, helping knockoffs to outperform less flexible methods like BH.

The whitening interpretation underscores not only the versatility of the fixed-$X$ knockoff filter, but also its broad applicability. Other than defining $\Sigma = (X^\tran X)^{-1}$, the above method is defined without reference to any design matrix and would be equally applicable if we simply observed $\hbeta \sim \cN_d(\beta,\Sigma)$ directly. In Section~\ref{sec:knockoffswithoutknockoffs} we push this interpretation further and consider how to apply the whitening method in generic statistical models with asymptotically Gaussian estimators of $d$ parameters, a ubiquitous problem encompassing generalized least squares, generalized linear models, quantile regression, and many other settings. 

\subsection{Which problems are hard for knockoffs?}

Having seen that the fixed-$X$ knockoff filter is an appealing and generally applicable method, it is natural to ask when we can expect it to outperform competitors like BH or Bonferroni. While this is a complicated question for a method as flexible as knockoffs, the whitening interpretation helps to identify an important vulnerability of the method: when the eigenstructure of $\Sigma$ is unfavorable, the price of whitening can be very steep, dooming the confirmatory analysis before the exploratory analysis even begins. One such example is the well-studied and seemingly innocuous {\em multiple comparisons to control} (MCC) problem \citep{dunnett1955multiple}.

\begin{example}[Multiple comparisons to control]\label{ex:mcc}
    Assume that we observe a continuous response on $m$ units under each of $d$ treatments, along with $m_0$ units under a control condition, and estimate an additive treatment effect. Let $y_{j,i}$ denote the $i$th response under treatment $j = 0, 1, \ldots, d$, with $j=0$ representing the control condition. We can pose this problem as a linear model by writing $ y_{0,i} = \beta_0 + \varepsilon_{0,i}$ for the control group, and
    \[
    y_{j,i} = \beta_0 + \beta_j + \varepsilon_{j,i}
    \]
    for the $j$th treatment group, so that $\beta_j$ is the differential effect of treatment $j$ relative to control. If we compose a vector $y = (y_0^\tran, y_1^\tran, \cdots, y_d^\tran)^\tran$ and model the errors as Gaussian with a common variance $\sigma^2$, we obtain a standard linear regression problem with intercept $\beta_0$ and can perform multiple testing on the hypotheses $H_j:\;\beta_j \neq 0$, for $j=1,\ldots,d$.
    
\end{example}

In the MCC problem, the OLS estimator of $\beta_j$ for $j \geq 1$ is
\[
\hbeta_j \;=\; \frac{1}{m}\sum_{i=1}^m y_{j,i} \,-\, \frac{1}{m_0}\sum_{i=1}^{m_0} y_{0,i},
\]
so that $\textnormal{corr}(\hbeta_j,\hbeta_k) = \rho = \frac{m}{m_0+m}$ for distinct $j,k>0$. If we normalize the design matrix appropriately, then $\Sigma$ is an equicorrelated matrix with diagonal entries equal to 1 and off-diagonal entries equal to $\rho$. Its first eigenvalue is $\lambda_1 = 1 + \rho (d-1) \geq \rho d$, with eigenvector $u_1 = \mathbf{1}_d/\sqrt{d}.$ As we will see in Section~\ref{sec:choosedelta}, the eigenstructure of $\Sigma$ makes it impossible to find $\Delta \succeq \Sigma$ without making most of the $\Delta_{jj}$ values enormous. Specifically, the $k$th smallest diagonal entry of any valid $\Delta$ matrix can be no smaller than $\rho k$.

Let $d_1$ denote the number of nonzero coefficients. As our measure of power, we use the {\em true positive rate} (TPR), defined as the expectation of the {\em true positive proportion} $\TPP = (R-V)/d_1$, the fraction of non-nulls discovered. $\TPP$ and TPR are meaningless if $d_1 = 0$, so we assume $d_1 \geq 1$ in all power calculations. 

In Section~\ref{sec:boundpower} we will show that in an MCC problem where the non-null proportion $d_1/d$ tends to a nonzero constant, no knockoff method can achieve a nontrivial TPR as $d\to \infty$ unless the nonzero coefficients grow roughly as $\sigma \sqrt{d}$; by contrast, the Bonferroni method has $\TPR \to 1$ when the nonzero coefficients grow at the rate $\beta_0 = \sigma \sqrt{2\log d}$. If all of the nonzero coefficients have size $\beta_0$ then, applying Theorem~\ref{thm:main}, the expected number of rejections of any fixed-$X$ knockoffs method at level $\alpha = 0.05$ is  no larger than 
\begin{equation}\label{eq:mccbound}
  \EE \,R \;\leq\; \frac{13\log d}{\rho} + 42.
\end{equation}

If $d=10^6$ and $\rho=0.5$, and the coefficients are all at the Bonferroni threshold $5.25\sigma$, the upper bound evaluates to $387$. This is a finite-sample result that applies even for a knockoff method designed by an analyst with full knowledge of the coefficient vector at every stage of the procedure, who can favor the non-null variables both by adding less noise to them and by ensuring they are always ordered first. We also show in Section~\ref{sec:boundpower} how to improve the bounds by simulation; in this example we can improve the bound from $387$ to $84$. Reducing $\rho$ from $0.5$ to $0.05$ only helps a little, increasing the simulation bound to $865$. Asymptotically, unless $\rho$ shrinks at the rate $\frac{\log d}{d}$, any knockoffs method will be have $\TPR \to 0$ for signals of size $\beta_0$.

By contrast, in a linear regression whose design matrix has positively equicorrelated columns, the situation is completely different. In that case, $\Var(\hbeta)$ is {\em negatively} equicorrelated, so its leading eigenvalue is small and all of the $\Delta_{jj}$ values can be $O(1)$. Likewise, if the design matrix is i.i.d. Gaussian with $n \geq 2d$, then for large $d$ we expect the largest eigenvalue of $\smash{(\frac{1}{n} X^\tran X)^{-1}}$ to be no larger than roughly $(1-\sqrt{1/2})^{-2} \approx 11.7$, per the Marchenko--Pastur distribution. Finally, if $\lambda_1$ is very large but $u_1$ is heavily concentrated on only a few variables, knockoffs may still perform well on the other variables. Section~\ref{sec:choosedelta} gives explicit and efficiently computable lower bounds for all order statistics of $\Delta_{11},\ldots,\Delta_{dd}$ as a function of a general matrix $\Sigma$.

The MCC setting is a prototypical example of a hard problem for knockoffs. More generally, if $\Sigma$ has large leading eigenvalues with dense eigenvectors then we must have $\Delta_{jj} \gg \Sigma_{jj}$ for most of the variables, rendering the test statistic $\sgn(\tbeta_j)$ nearly uninformative about $\beta_j$ even when the coefficient is very large. In short, we cannot whiten the estimator without burying the signal under a blizzard of noise.

\subsection{Related work}\label{sec:relatedwork}

Several prior works foreshadow the whitening interpretation of Section~\ref{sec:knockoffswhitened}. Most notably, an argument in the Supplement of \citet{barber2019knockoff} conditions on $(X+\tX)^\tran y$ and $|(X-\tX)^\tran y|$ to prove the knockoff filter controls the {\em directional FDR}, a more stringent version of FDR where the analyst must draw a conclusion about $\sgn(\beta_j)$ whenever they reject $H_j$. An analogous argument to theirs shows that the whitening method, viewed as a conditional post-selection inference procedure, controls the directional FDR as well. If we define the data-dependent one-sided null hypothesis $H_j^{\psi_j}:\;\psi_j \beta_j \leq 0$, then $\tp_1,\ldots,\tp_d$ are independent and valid conditional $p$-values for $H_1^{\psi_1},\ldots,H_d^{\psi_d}$ given $\xi$ and $|\tbeta|$. As a result, if we conclude $\sgn(\beta_j) = \psi_j$ whenever $H_j$ is rejected then we will control the directional FDR, both conditionally and marginally.

\citet{sarkar2021adjusting} also exploit the independence of $(X+\tX)^\tran y$ and $(X-\tX)^\tran y$ with a view toward developing new methodology. When $A$ is invertible, they view $D^{-1}(X-\tX)^\tran y$ and $\frac{1}{2}A^{-1}(X+\tX)^\tran y$ as two independent unbiased estimators of $\beta$, very similar to our data splitting interpretation. From this starting point they derive hybrid multiple testing procedures blending fixed-$X$ knockoffs with the BH method.

As a conditional post-selection inference method that uses a randomized data set for exploration, knockoffs is related to two other lines of work. In the statistics literature, the {\em data carving} method of \citet{fithian2014optimal} improves on the \citet{lee2016exact} post-selection inference method by performing a selection algorithm such as the lasso \citep{tibshirani1996regression} only on a subset of the data points, effectively randomizing the selection. \citet{tian2018selective} proposed explicitly randomizing the selection algorithm by adding Gaussian or other noise to the sufficient statistics used by the selection algorithm. Their work was inspired by a parallel literature on {\em adaptive data analysis} wherein randomized algorithms are used to answer many statistical queries about a database, one after the other, without allowing the analyst to ``overfit'' \citep{dwork2015preserving,dwork2015generalization}. More recently \citet{zrnic2020post} adapted a related approach based on algorithmic stability, also using randomization to preserve data for post-selection inference in a selected linear regression model.

Prior work on the knockoff filter's power has focused on positive results for specific implementations of the method. When studying the TPP-FDP tradeoff on the Lasso path, it is frequently assumed that a constant fraction of non-zero coefficients are sampled from a fixed distribution \citep{bayati2011lasso, su2017false}. In addition, it is often assumed that $n/p\to\delta$ for some positive constant $\delta$. Under this intermediate-dimensional regime, the theory of approximate message passing (AMP) \citep{bayati2011lasso} may be used to derive the asymptotic power of knockoff.
\citet{weinstein2017power} derives the asymptotic TPP-FDP tradeoff for a knockoff-inspired procedure that is only valid for i.i.d. covariates, while
\citet{weinstein2020power} and \citet{wang2020power} quantify the asymptotic power of knockoffs with the lasso coefficient difference test statistic.
\cite{liu2019power} studied the power of knockoffs under correlated design in the low dimensional setting, and showed that the knockoff filter has full asymptotic power when the precision matrix has vanishing diagonal entries. Going beyond the aforementioned linear sparsity assumptions, \citet{fan2020rank} studied the power of an oracle knockoff filter where the covariance structure of the variables is known. They assumed that the coefficients are fixed and relatively large, and showed under certain regularity conditions that the oracle knockoff filter is consistent. \cite{ke2020power} analyzed the phase diagram of the SDP knockoff, but their results are restricted to block-equicorrelated correlation structure with block size $2$.

Recently, \cite{spector2020powerful} showed that the equicorrelated and SDP knockoff methods can be asymptotically powerless for equicorrelated Gaussian design matrices with correlation $\rho\geq 0.5$, and propose a new method for generating knockoff variables, called minimum variance-based reconstructability (MVR) knockoffs, to resolve the issue. They showed that the TPR of the MVR knockoff converges to 1 under regularity conditions. The failure mode identified by \citet{spector2020powerful} is very different from the one we discuss here, which cannot be fixed within the knockoffs framework as currently defined. Section~\ref{sec:spectorjanson} discusses their results in light of the whitening interpretation. 

To the best of our knowledge, our results give the first universal upper bounds on what any feasible fixed-$X$ knockoff method can achieve. We prove finite-sample bounds on the maximum number of rejections that any knockoffs method can make, as a function of the design matrix $X$ and $\beta/\sigma$. Our bounds reveal that in some contexts the best achievable knockoff method markedly underperforms the Bonferroni method, but in other contexts off-the-shelf knockoff implementations outperform all other known methods. Understanding the fundamental limitations of the fixed-$X$ knockoff framework is a first step toward developing generalizations that ameliorate its flaws while retaining its many strengths.

\subsection{Outline of results}

Section~\ref{sec:bounds} derives both finite-sample and asymptotic bounds on the power of knockoffs as a function of the matrix $\Sigma$ and the coefficient vector $\beta$. Our argument proceeds in several steps. First, Theorem~\ref{thm:equivalence} in Section~\ref{sec:equivalence} proves a strong form of equivalence between the usual formulation of fixed-$X$ knockoffs as presented in Section~\ref{sec:knockoffsfirst} and the whitening formulation presented in Section~\ref{sec:knockoffswhitened}, establishing that every implementation of the former is also an implementation of the latter. Thus, it suffices to bound the whitening method's power. 

The main challenge in proving universal negative results is dealing with the analyst's ability to customize two stages of the procedure: in Stage 1, the analyst can choose any diagonal matrix $\Delta \succeq \Sigma$, and in Stage 2, the analyst can use any exploratory method to order the hypotheses and select the alternative directions. We address these two choices in sequence: in Section~\ref{sec:knockoffstar}, we show that an oracle analyst can optimize the exploration stage by sorting the variables in order of the log-odds $\eta_j = |2\beta_j\tbeta_j|/\sigma^2\Delta_{jj}$, and derive bounds on the expected number of rejections in terms of the number of $\eta_j$ values exceeding the critical threshold $-\log \alpha$. In Section~\ref{sec:choosedelta} we obtain lower bounds on the order statistics of $\Delta_{11},\ldots,\Delta_{dd}$. 

Theorem~\ref{thm:main} in Section~\ref{sec:boundpower} combines these two lines of argument to bound the expected number of rejections in terms of $\Sigma$ and $\beta/\sigma$. For the case where $\beta$ is sparse ($d_1/d \to 0$), Theorem~\ref{thm:main} is very optimistic since the analyst can ``cheat'' by giving up on the $d-d_1$ null indices while constructing $\Delta$. Theorem~\ref{thm:mainrandom} in Section~\ref{sec:boundpowerrandom} obtains a stronger bound for the more realistic setting where the analyst must choose $\Delta$ without knowing which hypotheses are non-null. Corollary~\ref{cor:randombeta} applies Theorems~\ref{thm:main}--\ref{thm:mainrandom} to prove, roughly, that if $\lambda_1 \gg \log d$ and $u_1$ is dense then all knockoff methods have $\TPR \to 0$ even in a regime where the Bonferroni correction enjoys $\TPR \to 1$. 

Since it is infeasible for a real analyst to achieve the same power as the oracle methods we describe, the bounds should not be taken as power estimates for real knockoff methods; instead, they reflect fundamental limits on the performance of the cleverest possible analyst, no matter how much domain wisdom or methodological wizardry they can bring to bear on the problem.

We emphasize that our results do not spell doom for knockoffs in all or even most settings; they hardly could, because there are indeed settings where knockoff methods outperform not only Bonferroni but also more powerful competitors like BH. Rather, they are a first step toward characterizing the regimes where each method is preferred and, we hope, toward developing methods that achieve near-optimal performance across all regimes. In particular, the regimes we identify can largely be recognized from the structure of $\Sigma$, which is known before the responses are observed. As a result, we can simply use a different method when the structure of $\Sigma$ is unfavorable.

Section~\ref{sec:knockoffswithoutknockoffs} delves deeper into the whitening method. Section~\ref{sec:implementwhitening} shows how we can implement it when we only observe an unbiased multivariate Gaussian estimator $\hbeta$ with no accompanying design matrix, and Theorem~\ref{thm:asymptoticnormal} in Section~\ref{sec:asymptoticnormal} shows we can do the same with an asymptotically normal estimator along with a consistent covariance matrix estimator, though the result is limited to the classical asymptotic regime. Section~\ref{sec:snp} applies Theorem~\ref{thm:asymptoticnormal} to analyze a nonparametric bootstrap example using stock market return data. Section~\ref{sec:simulations} shows empirical comparisons between knockoffs and competing methods in several regimes of interest, and Section~\ref{sec:discussion} concludes.

\section{Bounding the power of fixed-$X$ knockoffs}\label{sec:bounds}

\subsection{Equivalence of the two formulations}\label{sec:equivalence}

Next we will show that the two formulations of knockoffs in Sections \ref{sec:knockoffsfirst} and \ref{sec:knockoffswhitened}, which we will respectively call the {\em standard method} and the {\em whitening method}, are essentially equivalent. In the standard method, an implementation of the knockoff filter is fully defined by a valid knockoff matrix $\tX$ and a recipe for computing $W$-statistics satisfying the sufficiency and antisymmetry properties. In the whitening method, an implementation is fully defined by a diagonal matrix $\Delta \succeq \Sigma = (X^\tran X)^{-1}$ along with a recipe for computing a hypothesis ordering and $\psi_j$ values from $|\tbeta|$ and $\xi$. We will show that the implementations of each method are essentially in one-to-one correspondence with each other, using a coupling between $\XAug^\tran y$ and $\omega$.

To begin our analysis, assume that $\tX$ is a valid knockoff matrix with $X^\tran \tX = X^\tran X - D$, which can be constructed for any $D \preceq 2X^\tran X$ \citep{barber2015controlling}, and assume all $D_{jj} > 0$. Set $\Delta = 2D^{-1}$, so $A = X^\tran X - \Delta^{-1} = X^\tran X - \frac{1}{2}D$. 

Following a conditioning argument in \citet{barber2019knockoff}, if we add and subtract $\tX^\tran y$ from $X^\tran y$ we obtain a useful $2d$-variate Gaussian statistic whose mean and variance follow from $y \sim \cN_n(X\beta, \sigma^2 I_n)$:
\begingroup
\renewcommand*{\arraystretch}{1.5}
\begin{equation}
\label{eq:knockoffdistribution}
    \begin{pmatrix} (X+\tX)^\tran y \\ (X-\tX)^\tran y \end{pmatrix}
    \;=\; 
    \cN_{2d}\left(
    \begin{pmatrix} 2A\beta \\ D\beta \end{pmatrix},
    \;\;
    \begin{pmatrix} 4\sigma^2 A & 0 \\ 0 & 2\sigma^2 D \end{pmatrix}
    \right).
\end{equation}
\endgroup

It is suggestive that by rescaling the second component in \eqref{eq:knockoffdistribution} we obtain $D^{-1}(X - \tX)^\tran y \sim \cN_d(\beta, \sigma^2\Delta)$, which is the desired distribution for $\tbeta$. Pursuing this ansatz, set
\begin{equation}\label{eq:omegacoupling}
\omega \;=\; D^{-1}(X - \tX)^\tran y - \hbeta \;\sim\; \cN_d\left(0, \,\sigma^2 (\Delta - \Sigma)\right).
\end{equation}

Because $\omega$ is ancillary and $\hbeta$ is complete sufficient in the submodel where $\sigma^2$ is known, it follows by Basu's Theorem that $\omega$ is independent of $\hbeta$. As a result we have $\tbeta = D^{-1}(X - \tX)^\tran y$ and
\[    
\xi \;=\; X^\tran X\hbeta - \Delta^{-1}\tbeta \;=\; X^\tran y - \frac{1}{2}(X - \tX)^\tran y \;=\; \frac{1}{2}(X + \tX)^\tran y.
\]

Next we relate the whitening method's exploratory stage to a slight relaxation of the sufficiency and antisymmetry properties.
\begin{proposition}\label{prop:Wstar}
For $j = 1,\ldots, d$ define
\begin{equation}\label{eq:Wstar}
W_j^*(\XAug, y) = \sgn\left((X_j - \tX_j)^\tran y\right)\cdot W_j(\XAug, y).
\end{equation}

If $W$ satisfies the sufficiency and antisymmetry properties, then $W^*$ depends on $y$ only through the unordered pairs $\{X_1^\tran y, \,\tX_1^\tran y\}, \;\ldots,\; \{X_d^\tran y, \,\tX_d^\tran y\}$. In terms of the coupling defined by \eqref{eq:omegacoupling}, if $W$ satisfies the sufficiency and antisymmetry properties then $W^*$ depends on $y$ only through $\xi$ and $|\tbeta|$.
\end{proposition}

The first claim of Proposition~\ref{prop:Wstar} is proven in the supplement of \citet{barber2019knockoff} but we provide a full proof here for expository purposes.

\begin{proof}
By the antisymmetry property we have
\begin{align*}
W_j^*(\text{Swap}_j(\XAug), y) 
    &\;\;=\;\;  \sgn\left((\tX_j - X_j)^\tran y\right) \cdot W_j(\text{Swap}_j(\XAug), y)\\
    &\;\;=\;\;  \sgn\left(-(X_j - \tX_j)^\tran y\right) \cdot \left(-W_j(\XAug, y)\right)\\
    &\;\;=\;\; W_j^*(\XAug, y),
\end{align*}
so $W_k^*(\text{Swap}_j(\XAug), y) = W_k^*(\XAug, y)$ for all $k$ including $k=j$.

By the sufficiency property, $W$ only depends on $y$ through $\XAug^\tran y$, so the same is true of $W^*$. But we have just seen that $W^*$ is invariant to swapping $X_j$ with $\tX_j$, which amounts to swapping $X_j^\tran y$ with $\tX_j^\tran y$ without changing $\XAug^\tran \XAug$. In other words, $W^*$ depends on $y$ only through the unordered pairs $\{X_j^\tran y, \,\tX_j^\tran y\}$, for $j = 1,\ldots,d$.

For the coupling in \eqref{eq:omegacoupling}, we have $2\xi = (X + \tX)^\tran y$ and $D\tbeta = (X - \tX)^\tran y$. Hence, observing $\xi$ and $|\tbeta|$ is the same as observing $(X+\tX)^\tran y$ and $|(X - \tX)^\tran y|$, which is the same as observing the unordered pairs.
\end{proof}

We will say that $W$ satisfies the {\em unordered pair property} if $W^*$ depends on $y$ only through $\{X_j^\tran y, \tX_j^\tran y\}$ for $j=1,\ldots,d$. This condition relaxes the sufficiency and antisymmetry properties because it allows for $W^*$ to have unrestricted dependence on $X$ and $\tX$.  With these relationships established, we are now prepared to prove formal equivalence. 

\begin{theorem}\label{thm:equivalence}
Assume that $n \geq 2d$ and let $\tX$ be a knockoff matrix with $X^\tran \tX = X^\tran X - D$. Let $\Delta = 2D^{-1}$ and define $\omega$ as in \eqref{eq:omegacoupling}. Then
\begin{enumerate}[(a)]
\item For any implementation of the whitening method, we can construct $W$-statistics satisfying the unordered pair property so that the two methods give identical rejection sets.
\item For any $W$-statistics satisfying the unordered pair property such that $|W_1|,\ldots,|W_d|$ are almost surely positive with no ties, we can construct an implementation of the whitening method so that the two methods give identical rejection sets.
\end{enumerate}
\end{theorem}
\begin{proof}

For (a), take $W_{[j]} = (d + 1 - j) \,\cdot\, \psi_{[j]} \,\cdot\, \sgn(\tbeta_{[j]})$, for $j=1,\ldots,d$. Recalling that $\sgn(\tbeta_j) = \sgn((X_j - \tX_j)^\tran y)$, we have $W_{[j]}^* = (d + 1 - j) \,\cdot\, \psi_{[j]}$, so $W^*$ depends on $y$ only through $\xi$ and $|\tbeta|$, satisfying the unordered pair property.

For (b), because the method depends on $|W|$ only through the ordering of its coordinates, we can assume without loss of generality that $|W|$ is a permutation of $\{1,\ldots,d\}$. Take $[j]$ to be the index of the $j$th largest $|W_j|$ value, so that $|W_{[j]}| = d+1-j$, and set $\psi_j = \sgn(W_j^*)$, which is a function of $\xi$ and $|\tbeta|$. Then we again have $W_{[j]} = (d + 1 - j) \,\cdot\, \psi_{[j]} \,\cdot\, \sgn(\tbeta_{[j]})$.

To see why the two methods return the same rejection sets when $W_{[j]} = (d + 1 - j) \,\cdot\, \psi_{[j]} \,\cdot\, \sgn(\tbeta_{[j]})$, note first that $\tp_j = 1/2 \iff \sgn(\tbeta_j) = \psi_{j}  \iff W_j > 0$. As a result
\begin{align*}
    \hFDP_k^{\wh} 
\;=\; \frac{1 + \sum_{j=1}^d 1\{\tp_{[j]} = 1, \,[j] \leq k\}}{\sum_{j=1}^d 1\{\tp_{[j]} = 1/2, \,[j] \leq k\}}
    \;=\; \frac{1 + \sum_{j=1}^d 1\{W_{[j]} \leq -(d + 1 - k)\}}{\sum_{j=1}^d 1\{W_{[j]} \geq d + 1 - k\}}
    \;=\; \hFDP_{d+1-k}^{\kn}.
\end{align*}

For other values of $t \in (0, d+1]$, $\hFDP_{t}^{\kn} = \hFDP_{\lceil t\rceil}^{\kn}$, where $\lceil t\rceil$ is the integer ceiling of $t$. Therefore, $\hat{t} = d + 1 - \hk$ and the rejection sets are the same.

\end{proof}

Because the whitening method controls FDR, Theorem~\ref{thm:equivalence} implies that the unordered pair property is a sufficient condition for the standard knockoff filter to control FDR as well, relaxing the sufficiency and antisymmetry properties.  The requirement in (b) that all $|W_j|$ be positive and distinct is not really necessary; we could break ties or generate ``signs'' at random, or generalize the whitening method so that $\hFDP_k^{\wh}$ is only calculated for a subset of $k=\{1,\ldots,d\}$ corresponding to the indices where $|W_{[k]}|$ is positive and decreases. We ignore these generalizations for the sake of brevity because they can only reduce the number of rejections.

As Theorem~\ref{thm:equivalence} shows, any implementation of the standard method is achievable with the whitening method. It follows that any uniform bound on TPR for all implementations of the whitening method is also a uniform bound for all implementations of the standard method.

\subsection{The oracle analyst and the knockoff* procedure}
\label{sec:knockoffstar}

Throughout the rest of Section~\ref{sec:bounds}, we consider the perspective of an oracle analyst with full knowledge of $\beta$. To bound the power of this analyst we work backwards from the end of the procedure. In this section we discuss how the oracle analyst can make optimal decisions in the exploratory analysis of Stage 2, and in Sections~\ref{sec:choosedelta}--\ref{sec:boundpower} we apply our analysis to understand the ramifications for choosing $\Delta$.

The analyst's task in Stage 2 is to choose alternative directions $\psi_1,\ldots,\psi_d$ and a hypothesis ordering $[1],\ldots,[d]$, after observing $\xi$ and $|\tbeta|$. To choose $\psi_j$, recall that $\tp_j = 1/2$ only if $\sgn(\tbeta_j) = \psi_j$. As a result, an analyst who knows $\beta$ will always set $\psi_j = \sgn(\beta_j)$, unless $\beta_j = 0$ in which case $\psi_j$ makes no difference and can be chosen arbitrarily. Then the best possible ordering of variables is in decreasing order of the conditional log-odds:
\begin{equation}
\label{eq:etaj}
\eta_j \;=\; \logit \,\PP\left(\sgn(\tbeta_j) = \sgn(\beta_j) \;\big|\; \xi, |\tbeta|\right) \;=\;  
\frac{2|\tbeta_j| }{\sigma^2\Delta_{jj}} \cdot |\beta_j|.
\end{equation}

Because $\tbeta$ is continuous, the $\eta_j$ values for nonzero $\beta_j$ are almost surely positive and distinct, whereas if $\beta_j = 0$ then $\eta_j = 0$ too. Define the {\em knockoff* procedure} to be any method that sets $\psi_j = \sgn(\beta_j)$ if $\beta_j\neq 0$ and arbitrarily otherwise, and returns an ordering with
\[
\eta_{[1]} > \cdots > \eta_{[d_1]} > 0 = \eta_{[d_1+1]} = \cdots = \eta_{[d]}.
\]
The relative ordering of the null hypotheses is also arbitrary.
We show next that the knockoff* procedure maximizes the TPR among all implementations using the same $\Delta$.

\begin{restatable}{proposition}{knockoffstaropt}
\label{prop:knockoffstaropt}
Conditional on $\xi$ and $|\tbeta|$, both the TPP and the expected total number of rejections for the knockoff* procedure are stochastically larger than they are any other implementation with the same $\Delta$. Consequently, the knockoff* procedure maximizes both the TPR and the expected number of rejections.
\end{restatable}

While there is some ambiguity in defining the knockoff* procedure concerning the alternative direction and relative ordering for the null indices, the statement of Proposition~\ref{prop:knockoffstaropt} applies to any version of it, implying in particular that any two versions of knockoff* have the same TPR. The proof, given in Section~\ref{sec:proofs}, uses the following useful formula for the number of rejections $R$ as a function of $\hk$. Because we must have $\tp_{[\hk+1]} = 1$,
\begin{equation}\label{eq:Rformula}
\hFDP^\wh_{\hk} \leq \alpha < \hFDP^\wh_{\hk+1} \;\;\Longrightarrow\;\; 
1+\hk \,\leq\, (1+\alpha)R \,<\, 2+\hk \;\;\Longrightarrow\;\; 
R = \left\lceil \frac{1+\hk}{1+\alpha} \right\rceil.
\end{equation}

Next, we develop tools for bounding the power of the knockoff* procedure. If we write the order statistics of $\eta$ as $\eta_{(1)} > \cdots > \eta_{(d_1)} > 0 = \eta_{(d_1+1)} = \cdots = \eta_{(d)}$, then we can obtain upper bounds on the TPR of any knockoff procedure by bounding the knockoff* procedure which has $\eta_{[k]} = \eta_{(k)}$ for all $k$.

Roughly speaking, for knockoffs to make rejections, we must see $\tp_{[j]}=1/2$ at least $1/\alpha$ times as often as $\tp_{[j]}=1$ near the beginning of the list, so the largest $\eta_{(j)}$ values must exceed $-\log \alpha$; Proposition~\ref{prop:eta} gives a simple result relating these quantities. For $0 < p < q < 1$, define the random walk 
\[
    S_k^{p,q} = \sum_{j=1}^k p - Z_j, \quad \text{where } Z_1,Z_2, \ldots \simiid \text{Bern}\left(q\right),
\]
and let $M_d^{p,q} = \max\{k\leq d:\; S_k^{p,q} \geq 1-p\}$. 

\begin{proposition}
\label{prop:eta}
    Conditional on $\eta_{(1)} < -\log \alpha$, the number of rejections for any knockoff procedure at FDR level $\alpha$ is stochastically smaller than $M_d^{p,q_1}$ for $p=\frac{\alpha}{1+\alpha}$, $q_1=\frac{1}{1+e^{\eta_{(1)}}}$.
\end{proposition}

\begin{proof}
    Conditioning on $\xi$ and $|\tbeta|$, define $q_j = \frac{1}{1+e^{\eta_{(j)}}}$, and $Z_j' = 1\{\tp_j = 1\} \simind \text{Bern}(q_j)$. Then for the knockoff* procedure,
    \[
    \hFDP_k^\wh = \frac{1+A_k}{R_k}, \qquad \text{ where }  \;\;A_k = \sum_{j=1}^k Z_j', \quad \text{and } \;\;R_k = k-A_k = \sum_{j=1}^k 1-Z_j'.
    \]
    The number of knockoff* rejections is $R_{\hk} \leq \hk$, where $\hk$ is the largest index $k$ with
    \[
    \frac{1+A_k}{R_k} \;\leq\; \alpha \;\;\iff\;\; 1 \;\leq\; \alpha R_k - A_k \;=\; \sum_{j=1}^k \alpha - (1+\alpha)Z_j' \;=\; \frac{1}{1-p}\,\sum_{j=1}^k p - Z_j'.
    \]
    Next, define auxiliary random variables $B_1,\ldots,B_d \simiid \text{Bern}(q_1/q_j)$ and construct the random walk $S_k^{p,q_1}$ whose first $d$ increments are $Z_j = Z_j'\cdot B_j \simiid \text{Bern}(q_1)$. Then $\alpha R_k - A_k \leq \frac{1}{1-p} S_k^{p, q_1}$, so we have almost surely $\hk \leq M_d^{p,q_1}$.
\end{proof}

In particular, Proposition~\ref{prop:eta} implies that in any asymptotic regime where $\eta_{(1)} \toProb 0$, the number of knockoff rejections is $O_p(1)$ for any FDR level $\alpha$.

While Proposition~\ref{prop:eta} can help build our intuition about the role of the log-odds threshold $-\log\alpha$, it is not a strong enough result to control the power of an oracle analyst, who can always take $\Delta_{jj} \downarrow \Sigma_{jj}$ and force $[1] = j$ for a single promising variable $j$. The next result generalizes Proposition~\ref{prop:eta}, showing that knockoff* can expect no more than $O(k)$ rejections when there are fewer than $k$ large log-odds.

\begin{proposition}
\label{prop:bigokd}
    Fix a significance level $\alpha>0$ and margin $\delta > 0$. Conditional on $\eta_{(k)} < -\log (\alpha + \delta)$, the expected number of rejections for any knockoff procedure at FDR significance level $\alpha$ is upper bounded by $C_1(\alpha,\delta)k + C_2(\alpha,\delta)$, where $C_1$ and $C_2$ are constants that depend only on $\alpha$ and $\delta$.
    In particular, if $d_1 \to \infty$ and
    if
    \[
    \eta_{(\lceil cd_1 \rceil)} \toProb 0, \text{ for all } c>0,
    \]
    then $\TPR \to 0$ for any knockoff procedure at any FDR significance level.
\end{proposition}

We defer the proof of this proposition to Section \ref{sec:proofs}, where we also give explicit formulae for the constants $C_1$ and $C_2$. For example, when $\alpha=0.05$ and $\alpha + \delta = \sqrt{0.05}$, we have $C_1 \leq 2.3$ and $C_2 \leq 40$. These constants are not optimal and better bounds can be obtained for specific values $k,d,\alpha,\delta$ by simulating the random walk with $Z_1 = \cdots = Z_{k-1}=0$ and $\eta_{(k)} = \cdots = \eta_{(d)} = -\log(\alpha+\delta)$. 

Whereas Proposition~\ref{prop:bigokd} bounds what is possible after whitening and observing $\xi$ and $|\tbeta|$, the next section analyzes when it is impossible to carry out the whitening step without dramatic information loss that prevents the $\eta$ values from being large enough. We will show that this loss is determined by the eigen-decomposition of $\Sigma = (X^\tran X)^{-1}$, and characterize how large the signal must be to overcome this information loss and achieve nontrivial TPR, and illustrate our analysis with examples.

\subsection{Lower bounds for $\Delta$: how much noise must we add?}\label{sec:choosedelta}

Because the power of any knockoff procedure can be upper bounded by the knockoff* procedure from the previous section, understanding the best attainable power of knockoffs is a matter of investigating the joint distribution of the $\eta_1,\ldots,\eta_d$ values, which are independent of $\xi$ with distribution
\begin{equation}\label{eq:etadist}
\eta_j  \;=\; \left|\,\frac{2\beta_j\tbeta_j}{\sigma^2\Delta_{jj}}\,\right| \;\simind\; |\,\cN(\mu_j, 2\mu_j)\,|, \quad \text{ for } \mu_j = \frac{2\beta_j^2}{\sigma^2\Delta_{jj}}.
\end{equation}
The distribution of $\eta_j$ depends on the ratio of two quantities: the signal-to-noise ratio (SNR) $\beta_j/\sigma$, and the diagonal entry $\Delta_{jj}$. Since $\eta_j/\sqrt{2\mu_j} \sim |\cN(\sqrt{\mu_j/2}, 1)|$ is stochastically increasing in $\mu_j$, so is $\eta_j$. As a result, the {\em marginal} power of the knockoff* procedure is strictly increasing in each $\mu_j$. We will be primarily interested in the standardized case where $\Sigma_{jj} = \sigma^2 = 1$; in that case $\Delta_{jj} = \Var(\tbeta_j) = 1 + \Var(\omega_j)$, so knockoffs will (potentially) have power when the squared SNR is large compared to the whitening noise. 

Although most knockoff methods seek roughly to balance the whitening noise across the coordinates, if we want to bound the power of the best possible knockoff method we must consider the possibility that the analyst will favor some variables over others. For example, it is always possible to drive a single $\Delta_{jj} \to \Sigma_{jj}$, but only at the cost of driving $\Delta_{kk} \to \infty$ for all $k$ with $\Sigma_{jk} \neq 0$, since we must have
\[
\Var(\omega_k) \;\geq\; \frac{\text{Cov}^2(\omega_j,\omega_k)}{\text{Var}(\omega_j)} \;=\; \frac{\Sigma_{jk}^2}{\Delta_{jj}-\Sigma_{jj}} .
\]

Thus, while it is always possible to make any one $\Delta_{jj}$ small, it is sometimes impossible to make them all small at the same time. In this section we derive explicit lower bounds for the order statistics of $\Delta_{11},\ldots,\Delta_{dd}$ as functions of $\Sigma$, uniformly over all diagonal matrices $\Delta \succeq \Sigma$.
Let $\lambda_1$ be the leading eigenvalue of $\Sigma$, with eigenvector $u_1$. 
For any subset $S\subset \{1,\ldots,d\}$, we must have $\Delta_{S,S} \succeq \Sigma_{S,S}$, for the corresponding matrix blocks. As a result,
\[
\sum_{j\in S}\Delta_{jj}u_{1, j}^2 \;=\; u_{1, S}^\tran \,\Delta_{S,S}\,u_{1, S} \;\geq\;  \sum_{\ell =1}^d\lambda_{\ell}(u_{\ell, S}^\tran \,u_{1, S})^2 \;\geq\; \lambda_1 \|u_{1, S}\|_2^4.
\]
Therefore
\[
\|u_{1, S}\|^2_2 \max_{j\in S}\Delta_{jj}\;\geq\;\sum_{j\in S}\Delta_{jj} \,u_{1, j}^2 \;\geq\; \lambda_1\|u_{1, S}\|_2^4,
\]
and we have
\begin{equation}
\label{eq:dinverselowerbound}
    \max_{j\in S}\Delta_{jj} \;\geq\; \lambda_1\|u_{1,S}\|_2^2.
\end{equation}

To interpret equation~\eqref{eq:dinverselowerbound}, suppose that the entries of the unit vector $u_1$ are all at least as large as $c/\sqrt{d}$, for some $c \in (0, 1]$. Then for any block $\Delta_{S,S}$, we must have $\max_{j\in S} \Delta_{j,j} \geq c^2\lambda_1|S|/d$. Taking $S$ to be the set with the $k$ smallest entries, it follows that the $k$th smallest diagonal entry of $\Delta$ is at least $c^2\lambda_1 k/d$. Extending this argument yields explicit lower bounds for order statistics of $\Delta_{1,1}, \ldots, \Delta_{d,d}$, in terms of $\Sigma$.

\begin{proposition}\label{prop:Deltabound}
Let $\lambda_1 \geq \lambda_2 \geq \dots \lambda_d \geq 0$ be the eigenvalues of  $\Sigma$, with corresponding eigenvectors $u_1, \dots, u_d$, and let $u_{\ell,(1)}^2 \leq \cdots \leq u_{\ell, (d)}^2$ be the order statistics of $u_{\ell,1}^2, \ldots, u_{\ell,d}^2$, and let $\Delta_{(1,1)} \leq \cdots \leq \Delta_{(d,d)}$ be the order statistics of the diagonal entries of $\Delta$. Then for all $k=1,\ldots,d$, we can lower bound
\begin{equation}\label{eq:Deltabound}
 \Delta_{(k,k)} \;\geq\; b_k(\Sigma) \;=\; \max_{\ell \leq d} \,\lambda_{\ell} \sum_{j=1}^k u_{\ell, (j)}^2.
\end{equation}
\end{proposition}

\begin{proof}
 Note that in the argument for the lower bound \eqref{eq:dinverselowerbound}, we can just as well replace $\lambda_1, u_1$ with $\lambda_\ell, u_\ell$ for any $\ell=1,\ldots,d$. If $S$ is the index set for the $k$ smallest diagonal entries of $\Delta$, then
 \[
  \Delta_{(k,k)} \;=\; \max_{j\in S}\Delta_{jj} \;\geq\; \,\lambda_{\ell} \|u_{\ell, S}\|^2 \;\geq\; \,\lambda_{\ell} \sum_{j=1}^k u_{\ell, (j)}^2.
 \]
 Maximizing the lower bound over $\ell=1,\ldots,d$ gives the result.
\end{proof}

By construction, we have $b_1(\Sigma) \leq \cdots \leq b_d(\Sigma)$, and also $b_j(\Sigma) \geq \frac{j}{k} b_k(\Sigma)$ for all $j > k$. Continuing the example from above, we will have $b_k(\Sigma) \geq c^2 \lambda_1 k/d$. 

In the MCC case with $\lambda_1 = 1 + \rho(d-1)$ and $u_1 = \mathbf{1}_d /\sqrt{d}$, we have exactly $b_k(\Sigma) = \lambda_1 k /d \geq \rho k$ for all $k = 1,\ldots,d$. These bounds are nearly tight: for any size-$k$ subset $S$, the leading eigenvalue of $\Sigma_{S,S}$ is 
\[
\lambda_1^S \;=\; 1 + \rho(k-1) \;\leq\; \rho k + 1.
\]
Then for any $\epsilon > 0$, we can set 
\[
\Delta_{S,S} \;=\; (\lambda_1^S + \epsilon) I_k \;\succeq\; \Sigma_{S,S},
\]
but if we take $\epsilon \to 0$ we must have $\Delta_{jj} \to \infty$ for every $j \in S^\setcomp$. In other words, we can ``rescue'' our favorite $k$ variables from having $\Delta_{jj} > \rho k$, as long as we are willing to give up completely on the other $d-k$. 

\subsection{Uniform bounds on the power of fixed-$X$ knockoffs}\label{sec:boundpower}

Theorem~\ref{thm:main} combines Propositions~\ref{prop:bigokd} and \ref{prop:Deltabound} to obtain a finite-sample upper bound on the power of any fixed-$X$ knockoff procedure. We first give one more technical lemma:

\begin{restatable}{lemma}{Nalpha}
\label{lem:Nalpha}
Suppose that $z_i \sim \cN(\mu_i, 2\mu_i)$ for $i=1,2,\ldots$, where $0\leq \mu_i\leq \frac{-k\log \alpha}{2(k + i)}$, for some integer $k \geq 1$ and $\alpha \in (0,1)$. Then the expected number of $z_i$ values with $|z_i| > -\frac{1}{2}\log\alpha$ is bounded by $C_3(\alpha) k$.
\end{restatable}

Whereas Proposition~\ref{prop:bigokd} controls the number of rejections in terms of the number of large $\eta_j$ values, Lemma~\ref{lem:Nalpha} controls the number of large $\eta_j$ values we can get by chance from $\mu_j$ values falling below $-\frac{1}{2}\log\alpha$. We are now ready to prove our main result.

\begin{restatable}{theorem}{main}
\label{thm:main}
Let $\beta_{(1)}^2\geq \dots\geq \beta_{(d)}^2$ be the order statistics of $(\beta_1^2, \dots, \beta_d^2)$ and let $b_1(\Sigma), \ldots, b_d(\Sigma)$ be defined as in \eqref{eq:Deltabound}. For any target FDR level $\alpha>0$, let $k$ be the smallest integer such that
\begin{equation}
\label{eq:finitesamplecondition}
\frac{2\beta_{(k)}^2}{\sigma^2 b_k(\Sigma)} \;<\; -\frac{1}{2}\log\alpha,
\end{equation}
Then the expected number of rejections for any knockoff procedure at FDR level $\alpha$ is upper bounded by $C_1^*(\alpha)k + C_2^*(\alpha)$, where $C_1^*$ and $C_2^*$ depend only on $\alpha$. 
\end{restatable}

To sketch the proof, in expectation there are at most $(2 + C_3)k$ indices for which $\eta_j > -\frac{1}{2}\log\alpha$: $k$ with large $\beta_j$, $k$ more with small $\Delta_{jj}$, and $C_3 k$ more by chance; then applying Proposition~\ref{prop:bigokd} gives the final result.  As discussed in the proof of Theorem~\ref{thm:main}, if we replace $\beta_{(k)}$ with $\beta_{(1)}$ in \eqref{eq:finitesamplecondition} then we eliminate the portion with large $\beta_j$, so $2+C_3$ improves to $1+C_3$, giving correspondingly smaller bounds. Table~\ref{table:thmmain} shows the bounds at several $\alpha$ levels of interest.

\begin{table}[t]
\centering
\caption{Upper bounds from Theorem~\ref{thm:main} on $\EE R$ if $\displaystyle{\frac{2\beta_{(\ell)}^2}{\sigma^2 b_k(\Sigma)} \leq -\frac{1}{2}\log\alpha}$, given for $\ell = k$ and $\ell = 1$.}
\vspace{2mm}
\begin{tabular}{@{}cccc@{}}\toprule
 & $\alpha = 0.05$ & $\alpha = 0.1$ & $\alpha = 0.2$\\\midrule
Bound if $\ell = k$ & $7.0 k + 40$ & $10.6 k + 41$ & $19.1 k + 61$\\[5pt]
Bound if $\ell = 1$ & $4.7k + 43$ & $7.4k + 45$ & $14.3k + 66$\\

\bottomrule
\end{tabular}
\label{table:thmmain}
\end{table}

The dimension $d$ does not appear in the inequality \eqref{eq:finitesamplecondition}, unless it plays a role in determining $b_k(\Sigma)$. For MCC, $b_k(\Sigma) \approx \rho k$, so the number of rejections as $d\to\infty$ grows no faster than the square of the SNR. If we set the signal strength at the Bonferroni threshold $\beta_0 = \sigma \sqrt{2\log d}$, then solving \eqref{eq:finitesamplecondition} gives $k = -8\log d\,/\,\rho\log\alpha$, which we can plug directly into the bounds in the second row of Table~\ref{table:thmmain}. As a consequence, the TPR of any knockoffs method will tend to 0 if $d_1/\log d \to 0$. By contrast, if $d_1/d$ converges to a nonzero constant, then we must have $\beta_0/\sigma = O(\sqrt{d})$ to achieve nontrivial TPR in the limit.

More generally, to apply Theorem~\ref{thm:main} for a given covariance matrix $\Sigma$ at any signal strength, we can calculate the sequence $b_k(\Sigma)$ using the formula \eqref{eq:Deltabound}, and plug the smallest $k$ satisfying \eqref{eq:finitesamplecondition} into Table~\ref{table:thmmain}.

We can visualize the logic of Theorem~\ref{thm:main} and obtain improved bounds for specific examples by simulation. If all nonzero $\beta_j$ values are set to a common value $\beta_0$, then the $k$th largest $\mu_{(k)}$ is no larger than $2\beta_0^2/\sigma^2 b_k(\Sigma)$. To simulate the knockoff* procedure, we sample the $\eta_j$ values according to \eqref{eq:etadist}, sort them in descending order, and then simulate the resulting $\smash{\hFDP_k^\wh}$ process.  Figure~\ref{fig:simeta} illustrates two examples, with $d = 30,000$ and $d = 1,000,000$, and with $\beta_0$ set at the Bonferroni threshold $\sigma \sqrt{2\log d}$. We also set $d_1=d$ to simulate under the most optimistic conditions.\footnote{Alternatively we could choose a value for $d_1$ and set $\mu_{(k)} = 0$ for $k > d_1$, but it makes little difference if $\mu_{(d_1)} \ll -\log \alpha$.} We simulate the case $b_k(\Sigma) = 0.1 k/d$ and show $\mu_{(k)}$ along with a single realization of the $\eta_{(k)}$ and $\smash{\hFDP_k^\wh}$ processes. In addition we plot a histogram showing the distribution of the number of rejections for $\alpha = 0.05$.

\begin{figure}
    \centering
    \resizebox{\textwidth}{!}{
    \begin{tabular}{cc}
    \includegraphics[width=0.49\textwidth]{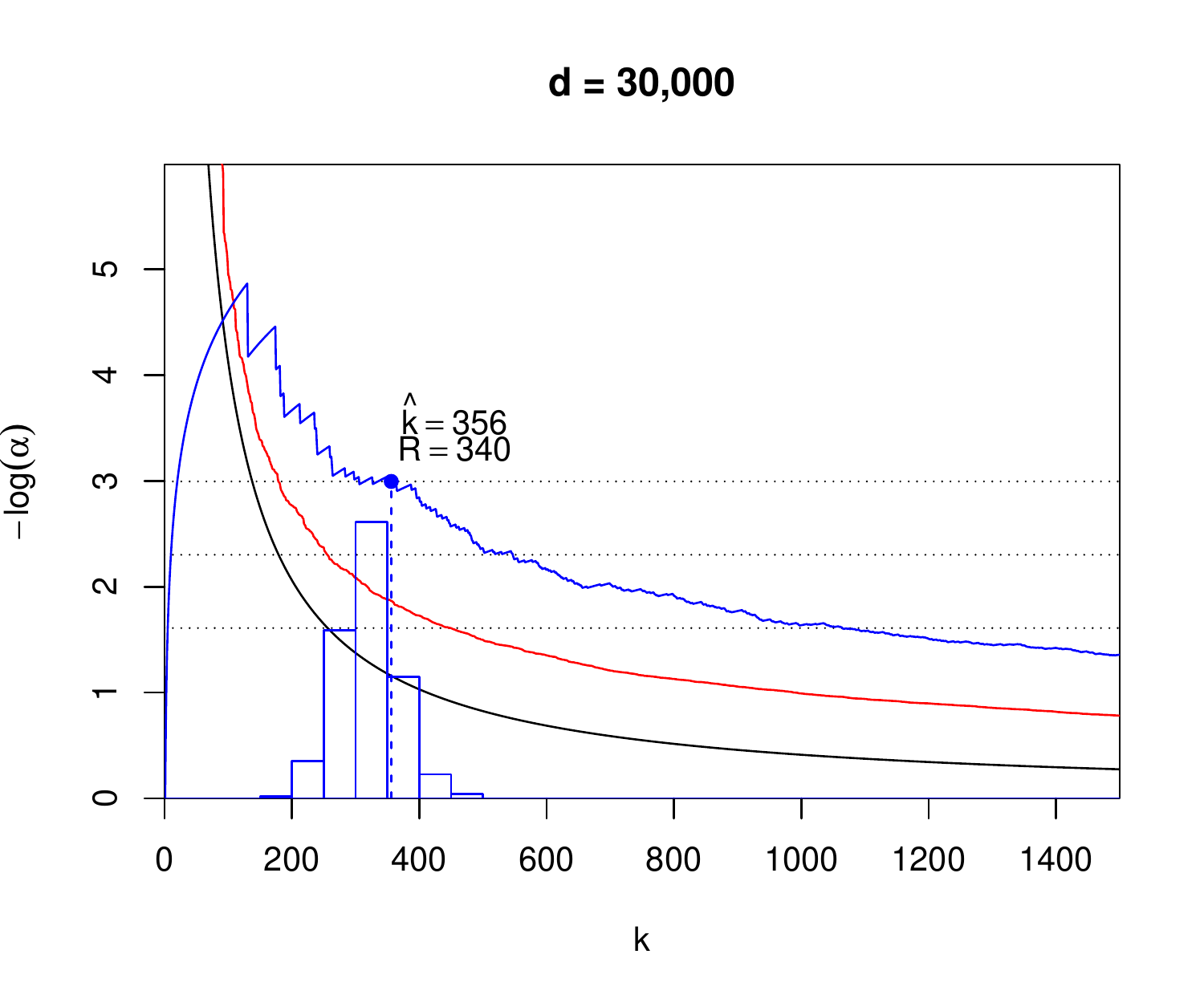} &
    \includegraphics[width=0.49\textwidth]{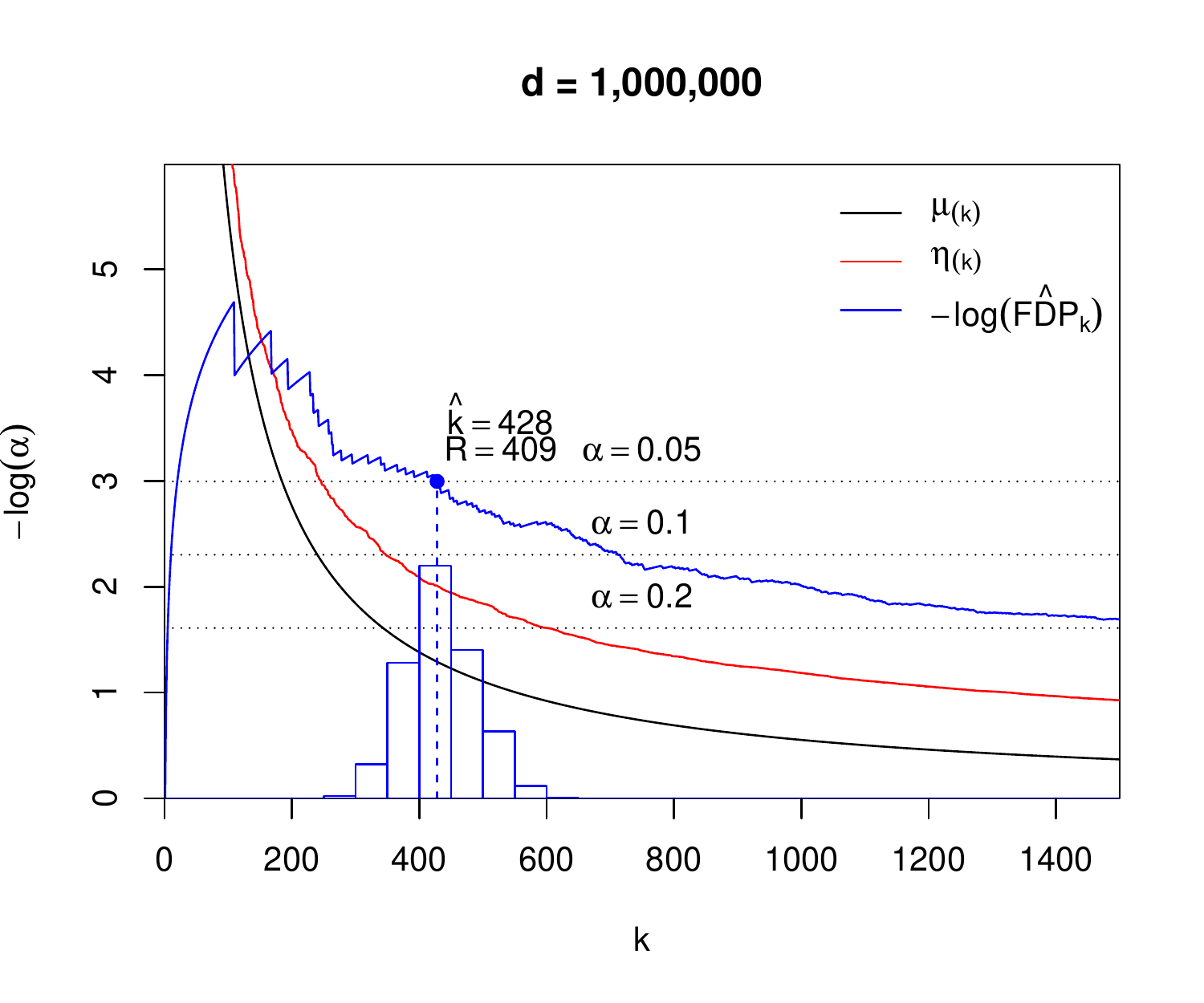} \\
    \end{tabular}
    }
    \caption{Simulation of the best possible knockoff method in the setting of Theorem~\ref{thm:main}, where the analyst has full knowledge of $\beta$ for both selecting $\Delta$ and carrying out the exploratory stage. We simulate under $b_k(\Sigma) = c^2 \lambda_1 k/d$, with $c^2 = 0.2$ and $\lambda_1 = d/2$, but these values can be substituted with $b_k(\Sigma)$ for any $\Sigma$ of interest. The signal strength is at the Bonferroni threshold, with all nonzero coefficients equal to $\beta_0 = \sigma \sqrt{2\log d}$, and $d_1 = d$ to be as optimistic as possible. The black line shows upper bounds for $\mu_{(k)} \leq 4\log d/b_k(\Sigma)$, the red line shows a single realization of $\eta_{(k)}$ where $\eta_j \simind |\cN(\mu_j, 2\mu_j)|$, and the blue line shows a realization of $-\log(\hFDP_k^\wh)$. The blue histogram shows the distribution of rejections over 1000 simulations.}
    \label{fig:simeta}
\end{figure}

Continuing the example from the previous section where $b_k(\Sigma) \geq c^2\lambda_1 k/d$, suppose that all nonzero coefficients take the same value $\beta_0$. In that case, Theorem~\ref{thm:main} implies that {\em no} knockoffs method can make more than $O\left(\frac{\beta_0^2 d}{\lambda_1}\right)$ rejections. If $\beta_0 = \sqrt{2r\log d}$ then no knockoffs method can make more than $O\left(\frac{d\log d}{\lambda_1}\right)$ rejections, so knockoffs must have $\TPR \to 0$ if $\lambda_1 \gg \frac{d\log d}{d_1}$. Thus, if $d_1/d$ converges to a nonzero constant, then $\TPR \to 0$ for knockoffs whenever $\lambda_1 \gg \log d$. Thus, knockoffs can underperform Bonferroni even in settings where the first eigenvalue asymptotically accounts for a vanishingly small fraction of the total variance. By contrast, if $d_1/d = o(\log d)$, then Theorem~\ref{thm:main} does not necessarily suggest we will have a problem. As we will discuss in the next section, Theorem~\ref{thm:main} gives especially optimistic predictions when $d_1 \ll d$ because the oracle analyst is too powerful.

\subsection{Stronger results with random non-null indices}\label{sec:boundpowerrandom}

While the bounds from Theorem~\ref{thm:main} are nontrivial, they are still highly optimistic, since they allow the analyst to optimize not only the exploratory stage (equivalent to choosing perfect $W$-statistics), but also the choice of $\Delta$ (equivalent to the choice of knockoff matrix $\tX$), with full knowledge of the coefficients. In particular, suppose that the true coefficient vector is sparse with $S_1 = \{j:\; \beta_j \neq 0\}$ and $S_0 = S_1^\setcomp$. Then the analyst loses nothing by sending $\Delta_{jj} \to \infty$ for all $j \in S_0$, effectively reducing the multiplicity of the problem to $d_1 \ll d$ by giving up on the null variables before the method has even begun. Whereas we might find it plausible that the exploratory stage is highly efficient, especially when the signals are very strong and the $W$-statistics are designed well, it is much less plausible that the analyst can discard all of the null variables even before observing any data. In any case, such a strategy is contrary to the spirit of knockoff methods that are actually in use, which seek to balance the $\Delta_{jj}$ (or $D_{jj}$) values, spreading power across all $d$ variables. 

We can force the oracle analyst into a more ``honest'' allocation of $\Delta_{jj}$ values by assuming the indices in $S_1$ are drawn uniformly at random from $\{1,\ldots,d\}$ only {\em after} $\Delta$ has been chosen. As before, we still assume $\beta$ is revealed to the analyst before the exploratory stage. Theorem~\ref{thm:mainrandom} gives a directly analogous result to Theorem~\ref{thm:main}, with a similar proof.

\begin{restatable}{theorem}{mainrandom}
\label{thm:mainrandom}
Let $\beta_{(1)}^2\geq \dots\geq \beta_{(d)}^2$ be the order statistics of $(\beta_1^2, \dots, \beta_d^2)$, and assume $\beta$ is a random permutation of the order statistics, independent of $\Delta$. Let $\pi_1 = d_1/d$ be the non-null proportion, and define $b_1(\Sigma), \ldots, b_d(\Sigma)$ as in \eqref{eq:Deltabound}. For any target FDR level $\alpha>0$, let $k$ be the smallest integer for which
\begin{equation}
\label{eq:finitesampleconditionrandom}
\frac{2\beta_{(k)}^2}{\sigma^2 b_{\lfloor k/\pi_1\rfloor}(\Sigma)} \;<\; -\frac{1}{2}\log\alpha,
\end{equation}
where $\lfloor\cdot\rfloor$ is the integer floor. Then the expected number of rejections for any knockoff procedure at FDR level $\alpha$ is upper bounded by $C_1^*(\alpha)k + C_2^*(\alpha)$, where $C_1^*$ and $C_2^*$ are the same constants as in Theorem~\ref{thm:main}. 
\end{restatable}

We can tighten the bounds if we replace $\beta_{(k)}^2$ with $\beta_{(1)}^2$ in \eqref{eq:finitesamplecondition}, just as with Theorem~\ref{thm:main}, so the bounds in both rows of Table~\ref{table:thmmain} apply. Likewise, we can simulate better bounds just as we did in Theorem 2, with the modification that after calculating $\mu_{(k)} = 2\beta_0^2/\sigma^2b_k(\Sigma)$, we randomly zero out $d-d_1$ of the $\mu_{(k)}$ values. Algorithm~\ref{alg:t3knockoff} fully specifies a Monte Carlo estimate for the highest achievable $\TPR$ when the analyst only knows the order statistics $\beta_{(1)}^2,\ldots,\beta_{(d)}^2$ when choosing $\Delta$.

\begin{algorithm}[t]
\setstretch{1.1}
\SetAlgoLined
\KwData{Covariance matrix $\Sigma$,\quad noise level $\sigma^2$, \quad FDR level $\alpha$,\quad order statistics $\beta_{(1)}^2,\ldots,\beta_{(d)}^2$}
\KwResult{Maximum achievable TPR for any knockoff method (Monte Carlo estimate)}
Calculate lower bounds $b_1(\Sigma),\ldots,b_d(\Sigma)$ using \eqref{eq:Deltabound}\;
\For{$m = 1$ \KwTo $M$} {
    Sample $\beta_1^2,\ldots,\beta_d^2$ by permuting order statistics\;
    \For{$j = 1$ \KwTo $d$} {
        Calculate $\mu_j = 2\beta_j^2/\sigma^2 b_j(\Sigma)$\;
        Sample $\eta_j \sim |\cN(\mu_j, 2\mu_j)|$\;
        Sample $\tp_j \in \{1/2, 1\}$ with $\logit \PP(\tp_j = 1/2) = \eta_j$\;
    }
    Order $\tp_{[1]},\ldots,\tp_{[d]}$ in descending order of $\eta_j$\;
    Find $\smash{\hk = \max\{k:\;\hFDP_k^\wh \leq \alpha\}}$\;
    Calculate $\TPP_m = \frac{1}{d_1}\sum_{j=1}^{\hk} 1\{\tp_{[j]} = 1/2 \text{ and } \beta_j \neq 0\}$\;
}
\KwRet{$\widehat{\TPR} = \frac{1}{M}\sum_m \TPP_m$\;}
\caption{T3-knockoff* simulation}
\label{alg:t3knockoff}
\end{algorithm}

By placing this mild limitation on the analyst's otherwise-total omniscience, we dramatically strengthen the result from Theorem~\ref{thm:main} in sparse problems. To illustrate, consider the bound $\EE R \leq 4.7k + 43$ from Table~\ref{table:thmmain}, the $\beta_{(1)}$ version of the bound as applied to $\alpha = 0.05$. If we were applying Theorem~\ref{thm:main}, then to get $\TPR \geq 1/2$, we would only need $k = d_1/10$ small $\Delta_{jj}$ values. By contrast, applying Theorem~\ref{thm:mainrandom} we would need $k/\pi_1 = d/10$ small $\Delta_{jj}$ values, since we can expect all but a fraction $\pi_1$ of them to be wasted on null indices. Pragmatically, even $d/10$ is not enough because the bounds in Table~\ref{table:thmmain} are not optimal; to achieve $\TPR \geq 1/2$ we realistically need most of the $\Delta_{jj}$ values to be small. 

Corollary~\ref{cor:t3knockoff} formalizes the sense in which Theorem~\ref{thm:mainrandom} and Algorithm~\ref{alg:t3knockoff} both upper-bound the best achievable knockoff method. Let $\cR_\alpha(\hbeta,\Sigma)$ denote any valid implementation of fixed-$X$ knockoffs at FDR level $\alpha$, mapping its inputs to a randomized rejection set. Then we have

\begin{restatable}{cor}{maximin}
\label{cor:t3knockoff}
    Fix $\sigma^2 = 1$. Let $B_T$ denote the bound obtained from Theorem~\ref{thm:mainrandom}, and $B_A = \EE \,\widehat{\TPR}$ is the output of Algorithm~\ref{alg:t3knockoff}; both bounds depend on $\Sigma$, $\alpha$, and the order statistics $\beta_{(1)}^2,\ldots,\beta_{(d)}^2$. Then
    \[
    \frac{1}{d_1} B_T \;\geq\; B_A \;\geq\; \sup_{\cR} \;\min_{\beta \in \Omega} \;\TPR(\cR, \beta),
    \]
    where $\Omega\subseteq\RR^d$ is the set of vectors $\beta\in\RR^d$ whose squared entries have the same order statistics, and $\TPR(\cR, \beta)$ is the $\TPR$ of method $\cR$ under $\hbeta \sim \cN_d(\beta,\Sigma)$. 
\end{restatable}

We finish by applying Theorems~\ref{thm:main}--\ref{thm:mainrandom} to prove a simple corollary classifying asymptotic regimes where knockoffs underperform Bonferroni. Whereas other works commonly assume a specific limiting distribution for the covariance matrix $\Sigma$ \citep[e.g.][]{liu2019power, weinstein2017power}, we only impose a less stringent condition quantifying how ``dense'' the first eigenvector $u_1$ is. For $c \geq 0$, define the distribution function of the entries of $u_1$, scaled by $\sqrt{d}$:
\[
F_d(c) \;=\; \frac{1}{d}\sum_{j=1}^d 1\left\{|u_{1,j}| \,\leq\, \frac{c}{\sqrt{d}}\right\}.
\]
We have been analyzing the case where $F_d(c) = 0$ for some $c>0$. For the next results, we will assume only that $F_d(\cdot) \to F(\cdot)$. If $\lambda_1 \gg \log d$, the following corollary gives relatively weak conditions for when the best possible knockoff procedure must have $\TPR \to 0$ for signals above the Bonferroni threshold.

\begin{restatable}{cor}{randombeta}
\label{cor:randombeta}
Consider a sequence of problems with $d \to \infty$ for which
\begin{enumerate}[(i)]
    \item 
    $F_d(\cdot)$ has a distribution limit $F(\cdot)$ with no point mass at 0, i.e. $F(0) = 0$.
    \item 
    For some $r>1$, $\beta_j = \sigma\sqrt{2r\log d}$ for a subset of $d_1$ indices, and $\beta_j=0$ for the rest.
    \item 
    {\em Either} $d_1 \to \infty$ and $\supp(\beta)$ is uniformly random, {\em or} $\lim\inf d_1/d > 0$ and $\supp(\beta)$ is arbitrary.
\end{enumerate}
Then 
\begin{enumerate}
    \item The Bonferroni procedure at any level $\alpha$ has $\TPR \to 1$, and
    \item If $\lambda_1/\log d \to \infty$ then any knockoff method at any level $\alpha$ has $\TPR \to 0$.
\end{enumerate}
\end{restatable}

Whereas negative results in prior works have focused on specific implementations of knockoffs under restrictive assumptions about the limiting distribution of the covariance matrix, typically under a polynomial sparsity regime ($d_1 \sim d^{1-\gamma}$ for $\gamma \in [0,1)$), Corollary~\ref{cor:randombeta} holds uniformly over all possible design choices made by the analyst, makes only mild limiting assumptions about the leading eigenvalue and eigenvector of $\Sigma$, and holds whenever $d_1 \to \infty$.

\section{Knockoffs for multivariate normal estimators}\label{sec:knockoffswithoutknockoffs}

\subsection{Implementing the whitening method}\label{sec:implementwhitening}

As we observed in Section~\ref{sec:knockoffswhitened}, the whitening method can be defined without reference to any design matrix $X$. If $\htheta$ is any estimator of a $d$-variate parameter $\theta$ with distribution $\htheta \sim \cN_d(\theta, \Sigma)$ and $\Sigma$ is known, or known up to a scalar multiplier, the whitening method can be applied directly. Theorem~\ref{thm:asymptoticnormal} in Section~\ref{sec:asymptoticnormal} shows that if $\htheta$ is only asymptotically Gaussian with a consistent estimator $\hSigma \toProb \Sigma$, then the whitening method can also be applied directly using $\hSigma$ as a plug-in estimator, and the FDR will be controlled asymptotically (though we note that the result is limited to classical fixed-dimensional asymptotics).

By shifting our focus to the estimator $\hbeta$, the whitening interpretation suggests easier solutions to closely related regression problems that do not easily lend themselves to introducing knockoff predictor variables. For example, if $\beta$ represents the population-risk-minimizing coefficients in a logistic regression, quantile regression, or robust regression with Huber loss, we can rely on well-established asymptotic normality results instead of trying to invent a novel framework for generating knockoff predictors in each new example.

Because of the one-to-one correspondence between methods in each formulation, we can even deploy state-of-the-art knockoff implementations, such as those in the \texttt{knockoff} package~\citep{patterson2017knockoff}, by constructing a pseudo-design matrix and pseudo-response using $\tbeta$ and $\xi$. 
\[
X^* = \binom{A^{1/2}}{\Delta^{-1/2}}, \quad \tX^* = \binom{A^{1/2}}{-\Delta^{-1/2}}\quad \text{ and }\;\; y^* = \binom{A^{-1/2}\xi}{\Delta^{-1/2}\tbeta}.
\]
Recalling that $A = \Sigma^{-1} - \Delta^{-1}$, and that $2\Delta^{-1}$ is the counterpart of $D$, a short calculation verifies that
\[
X^{*\tran} X^* \;=\; \tX^{*\tran} \tX^* = \Sigma^{-1}, \quad X^{*\tran} \tX^* = \Sigma^{-1} - 2\Delta^{-1}, \quad \text{ and } \;\;y^* \;\sim\; \cN\left(X^*\beta, \; \sigma^2 I_{2d}\right),
\]
so the distribution of the rejection set using $[X^*\, \tX^*]$ and $y^*$ as inputs to any off-the-shelf knockoffs package is the same as it would be if the problem had been posed as a linear regression to begin with. 

The whitening interpretation also suggests simpler methods like estimating $\beta$ and $\sigma^2$, and therefore also $\eta$, by regularized likelihood or Bayesian methods using the known Gaussian likelihood for $\xi$ and $|\tbeta|$. Because $\sgn(\tbeta_1),\ldots,\sgn(\tbeta_d)$ are missing, the log-likelihood is not necessarily concave, but the expectation-maximization (EM) algorithm can be used to impute the missing values. We defer these ideas to future work.

While the coupling in Section~\ref{sec:equivalence} gives a recipe for generating $\omega$ as a function of $\XAug$ and $y$, we can alternatively generate the noise directly using the residual variance $\hsigma^2$. We can even relax the usual dimension requirement that $n \geq 2d$, and require only that $n \geq d + r$ where $r = \text{rank}(\Delta-\Sigma) \leq d$. In that case let $M \in \RR^{d \times r}$ be any fixed matrix with $M M^\tran = \Delta - \Sigma$, and set
\begin{equation}\label{eq:genomega}
\omega = \sqrt{\frac{(n-d)v\hsigma^2}{\|\omega'\|^2}} \,\cdot\, M \omega', \quad \text{ for } \;\;\omega' \sim \cN_r(0, I_r) \;\;\text{ and }\;\; v \sim \text{Beta}\left(\frac{r}{2},\,\frac{n-d-r}{2}\right),
\end{equation}
or $v=1$ if $n = d+r$. Here $\omega'$ and $v$ are auxiliary random variables generated independently of the data and each other. Then $\omega$ is independent of $\hbeta$ because $\hsigma^2$ is, and Proposition~\ref{prop:noise} shows it has the desired distribution.

\begin{proposition}\label{prop:noise}
Define $\omega', v,$ and $\omega$ as in \eqref{eq:genomega}. Then $(n-d) v \hsigma^2 \sim \sigma^2\chi_{r}^2$, and $\omega \sim \cN_r(0, \sigma^2 (\Delta - \Sigma))$.
\end{proposition}
\begin{proof}
If $n = d + r$ and $v=1$ then $(n-d)v\hsigma^2 \sim \sigma^2\chi_{n-d}^2$ is immediate. If $n > d+r$, independently generate $a_1 \sim \chi_r^2 = \text{Gamma}\left(\frac{r}{2}, \,2\right)$ and $a_2 \sim \chi_{n-d-r}^2 = \text{Gamma}\left(\frac{n-d-r}{2}, \,2\right)$. Then it is a standard fact that the ratio $v = \frac{a_1}{a_1 + a_2} \sim \text{Beta}\left(\frac{r}{2}, \,\frac{n-d-r}{2}\right)$ is independent of $a_1 + a_2 \sim \chi_{n-d}^2$. As a result,
\[
v \,\cdot\, (n-d)\hsigma^2 \;\stackrel{\mathcal{D}}{=}\; v \,\cdot\, (a_1 + a_2)\sigma^2 \,=\, \sigma^2 a_1 \,\sim\, \sigma^2\chi_r^2.
\]
Similarly, $\omega'/\|\omega'\| \sim \text{Unif}(\mathbb{S}^{r-1})$ is independent of $\|\omega'\|^2 \sim \chi_r^2$. As a result,
\[
\omega^* = \sqrt{(n-d)v\hsigma^2} \cdot \frac{\omega'}{\|\omega'\|} \;\stackrel{\mathcal{D}}{=}\; \sqrt{\sigma^2\|\omega'\|^2} \cdot \frac{\omega'}{\|\omega'\|} \;\sim\; \cN_r(0, \sigma^2 I_r),
\]
and setting $\omega = M\omega^*$ gives the desired result.
\end{proof}

If we want the software package to generate the knockoff matrix for us, we can also do this as long as $r=d$. Then we can generate
\begin{equation}\label{eq:packageknockoffs}
X^* = \binom{\Sigma^{-1/2}}{0_{d\times d}}, \quad \text{ and }\;\; y^* = \binom{\Sigma^{-1/2}\hbeta}{\omega^*},
\end{equation}
where $\omega^*$ is defined as in Proposition~\ref{prop:noise}; then $X^{*\tran} X^* = \Sigma^{-1}$ and $y^* \sim \cN_{2d}(X^* \beta, \sigma^2 I_{2d})$ as desired.

If $n>d+r$, then part of $\hsigma^2$ is not used up in generating $\omega$:
\[
\tsigma^2 \;=\; \frac{(n-d)(1-v)}{n-d-r}\hsigma^2 \;\sim\; \frac{\sigma^2}{n-d-r}\,\chi_{n-d-r}^2.
\]
It is easily shown that $\hbeta, \omega,$ and $\tsigma^2$ are mutually independent, so allowing the analyst to use $\tsigma^2$ in Stage 2 has no effect on the conditional inference in Stage 3. 

The analogous quantity in the standard implementation is the residual variance from the augmented regression of $y$ on $\XAug$. Because $\tsigma^2$ is independent of $\XAug^\tran y$, the unordered pair property could be immediately relaxed to a requirement that $W^*$ be a function of the unordered pairs and $\tsigma^2$ without affecting the FDR control proof. We note \citet{chen2020prototype} also suggested the possibility of using this quantity. One use for $\tsigma^2$ could be as an input to generalized cross-validation~\citep{golub1979generalized}, and another could be as an aid in estimating $\beta$ and $\sigma^2$, as suggested above.

\subsection{Knockoffs for an asymptotically normal estimator}\label{sec:asymptoticnormal}

In this section we give simple sufficient conditions under which the whitening method can be applied to {\em any} asymptotically normal estimator $\htheta_n$ of a $d$-variate parameter $\theta\in \RR^d$, along with a consistent covariance estimator $\hSigma_n$. We consider a classical local asymptotic regime in which $\theta = \beta/\sqrt{n}$, for a local parameter $\beta\in \RR^d$ and a quantity $n$ that we can think of as a sample size, and assume\footnote{There is nothing special about $\htheta$ being $\sqrt{n}$-consistent; for example if we had $n^{1/4}(\htheta-\theta) \Rightarrow \cN(0,\Sigma)$ we could just as well define the local parameter $\beta = n^{-1/4}\theta$ and repeat the argument in this section.}
\[
\sqrt{n} \left(\htheta_n - \theta\right) \;\Rightarrow\; \cN_d(0, \Sigma), \quad \hSigma_n \toProb \Sigma.
\]

If we define $\hbeta_n = \sqrt{n}\,\htheta_n$, we have $\hbeta_n \Rightarrow \cN_d(\beta, \Sigma)$. It will be more convenient to work in terms of $\hbeta_n$ rather than $\htheta_n$; note that the null and directional alternative hypotheses are the same if we define them in terms of $\beta$ or $\theta$. For most implementations of knockoffs, there is a preprocessing step effectively standardizing either $\Sigma$ or $\Sigma^{-1}$ to have unit diagonal entries, in which case applying the method to $\hbeta_n$ or $\htheta_n$ would yield identical results.

We assume the analyst carries out knockoffs on $\hbeta_n$ according the usual whitening method, but plugging in $\hSigma_n$ for $\Sigma$. Let $\hDelta_n = \Delta(\hSigma_n) \succeq \hSigma$ denote the analyst's choice of diagonal matrix based on $\hSigma_n$. To generate the whitening noise, the analyst sets $\omega_n = \hM_n \omega'$, where $\hM_n = (\hDelta_n - \hSigma_n)^{1/2}$ and $\omega' \sim \cN_d(0, I_d)$ independently of the data.

Once the noise has been generated, the method proceeds exactly as described in Section~\ref{sec:knockoffswhitened}. That is, for
\[
\tbeta_n = \hbeta_n + \omega_n, \quad \xi_n = \hSigma_n^{-1} \hbeta_n - \hDelta_n^{-1} \tbeta_n,
\]
the analyst chooses alternative directions and an ordering based on $\xi_n, |\tbeta_n|$ and then carries out Selective SeqStep on the binary $p$-values.

To impose some regularity on the exploratory analysis, we consider knockoff methods for which the exploratory stage can be defined in terms of $d$ non-negative random variables $W_j^+(\xi, \tbeta)$, along with
\begin{equation}\label{eq:Wjplus}
W_j^-(\xi, \tbeta) \;=\; W_j^+(\xi, \text{flip}_j(\tbeta)), \quad \text{ and } \;\; W_j \;=\; \max\left\{W_j^+, W_j^-\right\} \cdot \sgn(W_j^+-W_j^-),
\end{equation}
where $\text{flip}_j:\;\RR^d \to \RR^d$ flips the sign of its $j$th argument and leaves the others fixed. The $W$-statistics given in \eqref{eq:Wjplus} are then used the same way the $W$-statistics are used in the standard formulation of knockoffs: the variables are ordered by $|W_j|$ with $\psi_j = \sgn(W_j) \cdot \sgn(\tbeta_j)$, so $\tp_j = 1/2 \Leftrightarrow W_j > 0$. It is easy to see that 
$W_j^* \;=\; W_j \cdot \sgn(\tbeta_j)$ is a function of $\xi$ and $|\tbeta|$ alone, so any statistic $W^+$ defines a valid implementation of the whitening method. This way of defining knockoffs is inspired by the ``$\lambda$ signed max'' statistic, which we can recover by setting
\[
W_j^+ = \max\{\lambda > 0:\; \hbeta_j^\lambda  \neq 0\},
\]
where $\hbeta^\lambda$ is some regularized estimator of $\beta$ with penalty parameter $\lambda>0$, possibly in an augmented linear regression model for the response $y^*$ against the design matrix $[X^*\,\tX^*]$ as defined above.

Let $\cR_\alpha(\hbeta, \,\Sigma; \,\omega)$ denote the $\alpha$-level rejection set for the knockoff method defined by $\Delta(\cdot)$ and $W^+(\cdot)$, as it would be applied to $\hbeta \sim \cN_d(\beta, \Sigma)$ with whitening noise $\omega$; we will typically suppress the randomization variable $\omega$. We say $\cR$ is a {\em continuous knockoffs procedure} if $\Delta(\cdot)$ and $W^+(\cdot)$ are continuous functions of their arguments. In the fixed asymptotic regime under mild conditions, the rejection set converges to its distribution under exact Gaussian sampling, as we show next.

\begin{restatable}{theorem}{asymptotic}
\label{thm:asymptoticnormal}
    Assume that we carry out a continuous knockoffs procedure $\cR(\cdot)$ on the asymptotically normal estimator
    \[
    \hbeta_n \;\Rightarrow\; \cN_d(\beta, \Sigma), \quad \text{ and } \;\hSigma_n \;\toProb\; \Sigma,
    \]
    where $\Sigma$ is nonsingular. Let
    \[
    E = \{(\xi, \tbeta):\; W_1^+, W_1^-, \ldots, W_d^+, W_d^- \textnormal{ are all positive and distinct}\},
    \]
    and assume that $E^\setcomp$ has Lebesgue measure zero. 
    Then we have for any $\alpha$,
    \[
    \cR_\alpha\left(\hbeta_n, \,\hSigma_n; \,\omega_n\right) \;\Rightarrow\; \cR_\alpha\left(\hbeta, \,\Sigma; \,\omega\right),
    \]
    where $\hbeta \sim \cN_d(\beta,\Sigma)$ and $\omega \sim \cN_d(0,\Delta(\Sigma)-\Sigma)$ independently. In particular, the FDR and TPR of $\cR_\alpha$ converge to their values under exact normal sampling.
\end{restatable}

As an immediate application of Theorem~\ref{thm:asymptoticnormal}, we see that continuous knockoff methods can be applied to any asymptotically normal maximum likelihood estimator.

\begin{example}[Maximum likelihood estimation]\label{ex:mle}
    Suppose we observe a sample $X_1,\ldots,X_n$ from a parametric family $\cP = \{p_\theta:\; \theta\in\Theta\subseteq \RR^d\}$, where $p_\theta$ is a density with respect to a common dominating measure $\mu$. Define the maximum likelihood estimator (MLE)
    \[
    \htheta_n \;=\; \argmax_{\theta\in \Theta} \sum_{i=1}^n \log p_\theta(X_i)
    \]
    Assume further that $\cP$ is sufficiently regular so that for $\theta$ in the interior of $\Theta$,
    \[
    \sqrt{n}\left(\htheta_n - \theta\right) \;\Rightarrow\; \cN_d\left(0, J(\theta)^{-1}\right),
    \]
    where $J(\theta)^{-1}$ is the Fisher information at $\theta$. If $J(\theta)$ is nonsingular and continuous at $\theta$, and $\cR(\cdot)$ is continuous, then by Theorem~\ref{thm:asymptoticnormal}, $\cR_\alpha(\htheta_n, \,J(\htheta_n)^{-1})$ asymptotically controls the FDR at level $\alpha$.
\end{example}

One interesting special case of Example~\ref{ex:mle} is the logistic regression model where we observe predictors $x_i\in\RR^d$ and binary response $y_i$, and model $\logit \PP(y_i = 1 \mid x_i) = x_i^\tran \beta$ for some $\beta \in \RR^d$. Then, just as in linear regression, we have a design matrix $X \in \RR^{n\times d}$ and response $y \in \{0,1\}^n$. If we think of knockoffs in terms of the standard formulation, we might be tempted to try generalizing it by adding some matrix $\tX$ of predictor variables and fitting a logistic regression on the augmented matrix $\XAug$. But it is unclear how to design such a matrix {\em a priori}, since the MLE $\hbeta_n$ solves the weighted least squares problem
\[
\hbeta_n \;=\; \left(X^\tran \hV_n X\right)^{-1} X^\tran \hV_n y, \quad \text{ where } \; \hV_n \;=\; \diag\left(\Var_{\hbeta_n}(y_i)\right).
\]
In particular, it is the predictor variables' {\em weighted} correlations with each other that determine the asymptotic distribution of $\hbeta_n$, not the unweighted correlations as in linear regression. But to design a knockoff predictor $\tX_j$ whose weighted correlations with $X_{-j}$ are the same as those of the real predictor $X_j$, we would need to know $\hbeta$, which determines the weights, ahead of time.

By contrast, if we view the MLE $\hbeta$ for linear regression as an asymptotically Gaussian random vector, we can simply treat it just as we would the OLS estimator for a linear regression, and treat the variance estimator $\hSigma_n = (X^\tran V_{\hbeta_n} X)^{-1}$ as a consistent plug-in estimator for the asymptotic variance.

We caution, however, that asymptotic normal approximations can be unreliable in high-dimensional settings, particularly when $d/n \to 0$. See e.g. \citet{el2013robust} and \citet{donoho2016high} regarding the limitations of asymptotic normal approximation in high-dimensional M-estimation; \citet{el2018can} also cast doubt on the accuracy of the bootstrap in high dimensions. In particular \citet{sur2019modern} showed that logistic regression is unreliable in high dimensions even with five to ten observations per predictor variable. As a result, we recommend applying this method only in settings where $n \gg d$ so the multivariate Gaussian approximation is reliable; determining the full range of application beyond the classical setting is outside the scope of this work.

\subsection{Example: Analysis of stock market data}\label{sec:snp}

We now use an example from stock market data to illustrate the implementation of fixed-X knockoff when the test statistics are asymptotically multivariate Gaussian. Our goal is to test whether the daily returns of a given stock are predictable based on the one-day-lagged returns of the overall market, as measured by the Standard and Poors (S\&P) 500 index. We use S\&P 500 data from February 2013 to February 2018\footnote{Source: https://www.kaggle.com/camnugent/sandp500}, totalling 1259 trading days, and restrict our analysis to the $470$ stocks that belonged to the S\&P 500 over the entire period considered. For each stock, we test whether its excess returns relative to the market are correlated with the one-day-lagged market returns.

For $j=1,\ldots,d=470$, let $y_j(t)$ denote the price of stock $j$ on day $t$, and let $y_0(t)$ denote the level of the S\&P 500 index on day $t$. Then for $j=0,1,\ldots,470$, the return for stock $j$ on day $t$ is
\[
r_j(t) \;=\; \frac{y_j(t)}{y_{j}(t-1)} \,-\, 1.
\]
Our parameter of interest for stock $j$ is the Spearman (rank) correlation of its excess return, relative to the market, with the one-day-lagged market return:
\[
\rho_j \;=\; \text{Corr}_S(e_j(t), \, r_0(t-1)), \quad \text{ where } \;\; e_j(t) = r_j(t) - r_0(t),
\]
and  $\text{Corr}_S$ denotes the Spearman correlation. Roughly speaking, if $\rho_j > 0$ then stock $j$ is a ``good bet'' today if the market performed well yesterday and a ``bad bet'' otherwise, and if $\rho_j < 0$ the reverse is true. Because economic theory suggests stock returns should be unpredictable, we expect most of the true correlations to be close to 0. We will test $H_j:\; \rho_j = 0$ against the two-sided alternative for $j = 1,\ldots,d$.

We define the problem in terms of  excess returns $e_j(t)$ instead of ``raw'' returns $r_j(t)$ for two reasons: First, predictability of excess returns is more easily translated into a trading strategy, since an investor could hedge their bets against the market, giving another reason for most $\rho_j$ to be small. Second, the pairwise correlation between stock returns for any two stocks is usually positive, since good or bad news about the economy at large tends to move all stocks in the same direction. As a result, the leading eigenvalue and eigenvector of $\text{Var}(r(t))$ mimics the leading eigenvalue and eigenvector for MCC, giving knockoffs no chance to perform well. The common market factor is so strong, accounting for almost half the total variance, that it is doubtful whether FDR is even the right error rate to control \citep{efron2007correlation,schwartzman2011effect,kluger2021central}. By contrast, the excess returns do not exhibit the same sort of positive correlation because we have removed it by subtracting off the market index return.

To test the hypotheses, we calculate the sample Spearman correlations $\hrho_1,\ldots,\hrho_d$, and use the bootstrap to obtain a nonparametric estimator $\hSigma$ for the covariance matrix $\Sigma = \Var(\hrho)$; which has similar structure to $\Var(e(t))$. Because the autocorrelations of $r_j(t)$ are small, we use the i.i.d. bootstrap, but we obtain similar rejection sets if we use the block bootstrap for block length 15 or 50 (blocks of three or ten weeks respectively). The bootstrap resampling is repeated 10,000 times. By inspecting the normal Q-Q plot of the sample Spearman correlation across bootstrapped samples, we find that the asymptotic Gaussian approximation fits the bootstrap distribution reasonably well even in the tails where rejections occur, at least marginally for each stock. We use the formulation in the previous section to implement the fixed-$X$ knockoff treating $\hrho$ as an asymptotically Gaussian estimator of $\rho$. We use the lasso $\lambda$ signed-max method with SDP knockoffs. 

The leading eigenvalue of the bootstrap correlation matrix $\widehat{\text{Corr}}(\hrho)$ is 53, accounting for about 11\% of the total variance, and the entries of the leading eigenvector are dense but roughly symmetrically distributed around zero, ranging from $-0.15$ to $0.15$. Thus, the eigenstructure does not rule out a well-designed implementation of knockoffs performing well but it may nonetheless be worrisome.

Given the estimated correlation matrix, we use the procedure in Section \ref{sec:mvnknockoff} to implement the fixed-X knockoff. Since this requires generating additional Gaussian noise, we repeat the procedure 30 times to obtain more stable rejection sets, and report hypotheses that are rejected for at least 50\% of times. This stability selection criteria is similar to that proposed in \citep{ren2020derandomizing}. We found that at FDR level $\alpha = .2$, knockoff obtains no rejections for 26 out of 30 trials. The other 4 times, there are 7 rejections on average, and the stocks BA, BRK.B, MDLZ and UAA are always rejected. As such, no stocks meet the stability selection threshold. By contrast, both the BH and the Bonferroni correction makes a number of rejections, with the BH making 11. The stocks rejected by the BH and the Bonferroni correction are summarized in Table \ref{table:snp500}.

Figure \ref{fig:wstatistics} shows the $W$-statistics and along with $\hFDP_k$ from a representative trial. We find that there are many negative $W$ statistics in the front of the list, indicating that knockoff will not have many rejections even if the significance level $\alpha$ is reasonably relaxed. We see in Section~\ref{sec:poweranalysis} that we could have anticipated this underperformance by inspecting the $\Delta$ values.

\begin{table}[t]
\centering
\caption{Rejections of different methods at different target FDR levels, for testing the null hypotheses that the returns of each stock are correlated with the lag-1 return of the S\&P 500 index. ``Additional rejections" are rejections that are made at $\alpha=0.2$ but not $\alpha=0.1$.}
\vspace{2mm}
\begin{tabular}{@{}ccc@{}}\toprule
Method & Rejections at $\alpha=.1$ (stock tickers) & Additional rejections at $\alpha=.2$ \\\midrule
BH & 7 (AVGO, BA, BRK.B, JEC, MOS, SWKS, UAA) & 4 (FISV, KORS, V, WAT)\\[5pt]
Bonferroni & 3 (BRK.B, JEC, MOS) & 0\\[5pt]
Fixed-$X$ knockoff & 0 & 0 \\
\bottomrule
\end{tabular}
\label{table:snp500}
\end{table}

\begin{figure}
    \centering
    \begin{tabular}{cc}
    \includegraphics[width=0.85\textwidth]{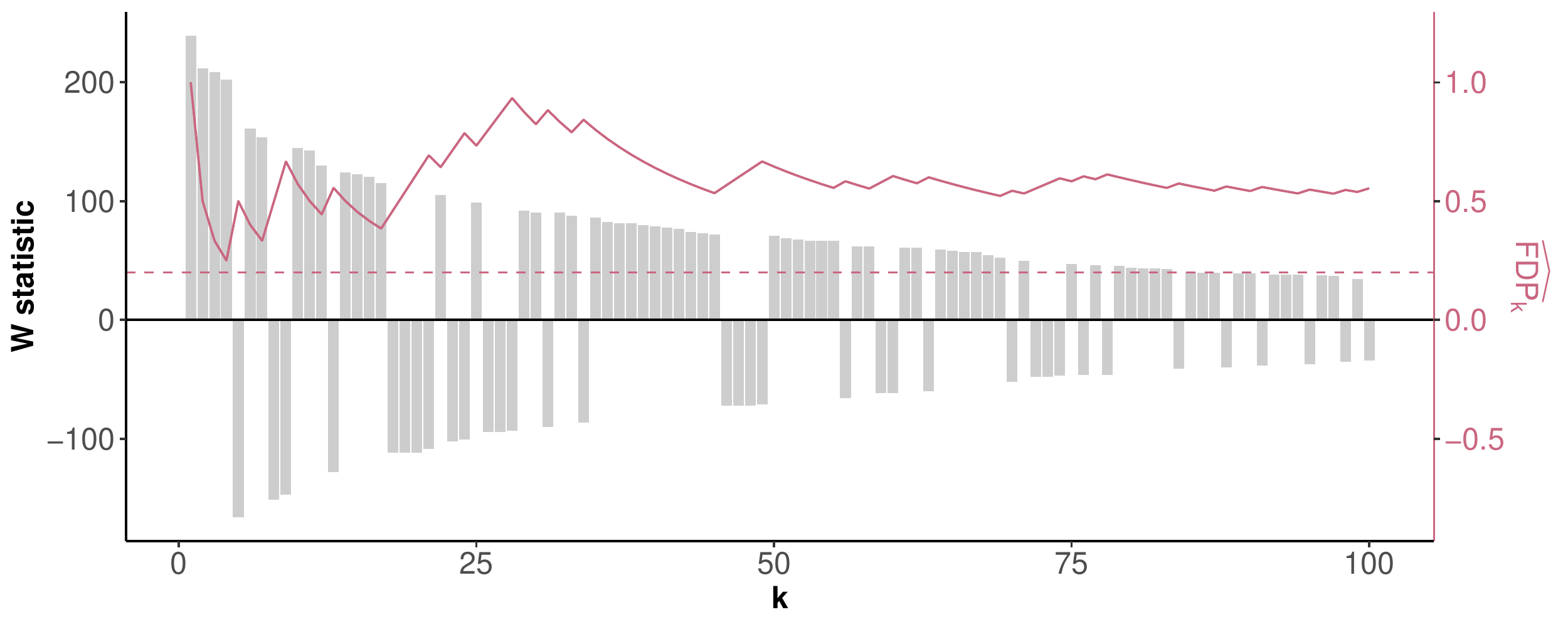}  
    \end{tabular}
    \caption{The 100 largest $W$ statistics of the S\&P 500 example, ranked by absolute value, and the estimated FDP $\hFDP_k$ from the Selective Seqstep procedure. Even if we use the more liberal FDP estimator without the ``$1+$,'' as plotted here, knockoffs is unable to make rejections at $\alpha = 0.2$, represented by the horizontal line.}
    \label{fig:wstatistics}
\end{figure}

\section{Numerical results}\label{sec:simulations}

\subsection{Fixed knockoff for Gaussian linear model}
In this section, we simulate under the Gaussian linear model to compare scenarios where the TPR of the best achievable knockoff method is close to zero against similar scenarios where knockoff methods perform well. For the design matrix $X \in \RR^{n\times d}$, we generate random matrices whose rows are generated i.i.d from $N(0, K), K \in \RR^{d\times d}$. We consider the following two regimes:
\begin{enumerate}[(a)]
    \item \textbf{Positively equi-correlated OLS estimator} $K^{-1}$ is an equicorrelation matrix with correlation $\rho = 0.2$, i.e. $K^{-1}_{ii}=1, 1\leq i\leq d$, and $K^{-1}_{ij} = 0.2, 1\leq i < j\leq d$;
    \item \textbf{Positively equi-correlated covariates} $K$ is an equicorrelation matrix with correlation $\rho = 0.2$.
\end{enumerate}
  In regime (a), the columns of $X$ are negatively correlated and the OLS test statistics are positive correlated, and we expect that any knockoff methods will have trivial power. In regime (b), the columns of $X$ are positively correlated and the OLS test statistics are negatively correlated.

We choose $n = 3000$ and $d = 1000$. For each $K$, we fix one realization of the random matrix, and then normalize its columns to obtain the design matrix $X$.
Next, we generate the response $y \in \RR^{d}$ as follows. First, to define $\beta$, we choose $s = 30$ coefficients uniformly at random and let $\beta_j = 5$ for each of the selected coefficients. We then generate $y = X\beta + \varepsilon$, where $\ep_i$ are i.i.d standard normal errors. The above data generating procedure is repeated 600 times for each $K$.

For each design matrix $X$, we generate the knockoff matrix $\tX$ using the MVR-knockoff \citep{spector2020powerful} and SDP-knockoff algorithms \citep{barber2015controlling}. For each instance of $(X, \tX, y)$, we first consider the knockoff* procedure introduced in Section \ref{sec:knockoffstar}, which is the best achievable knockoff method once the knockoffs have been constructed, but is infeasible since it can only be carried out with knowledge of the true coefficients $\beta$. We also consider a practically feasible knockoff method which uses the maximum lasso penalty level as test statistic.

In addition, we include the T3-knockoff* ``procedure'' from Section~\ref{sec:boundpowerrandom}, a simulation-based bound of any knockoff method under the assumptions of Theorem~\ref{thm:mainrandom}, where the analyst is not allowed to know the nonzero indices of $\beta$ at the time of determining the knockoff matrix $\tX$ (or equivalently, determining $\Delta$). $\beta$ is still revealed to the analyst immediately after the knockoffs are generated so the analyst can carry out the knockoff* method. T3-knockoff* is defined in Algorithm~\ref{alg:t3knockoff} and uniformly upper bounds the power achievable by SDP-knockoffs or MVR-knockoffs in any setting where $\beta$ is random with an exchangeable distribution over the $d$ indices, so it is more optimistic than SDP-knockoff* or MVR-knockoff*.

Finally, we consider the BH procedure and the Bonferroni test on OLS $p$-values for baseline comparison. Note that the BH procedure is not theoretically guaranteed to control the FDR at the desired level unless the columns of $X$ are orthogonal.

We obtain the SDP-knockoff matrix using the \texttt{knockoff} package~\citep{patterson2017knockoff}, and the MVR-knockoff using the \texttt{knockpy} package~\citep{spector2020powerful}. For both knockoff matrices, we then compute the the maximum lasso penalty level statistics with the \texttt{knockoff} package.
Overall, the performance of knockoff shows a stark contrast under these two regimes, while the performance of BH and Bonferroni appear to be much more stable. Figure \ref{fig:knockoff_power_example} shows the power of the BH, the Bonferroni and the SDP-knockoff tests. For reference, we also include the case where the design matrix is i.i.d. Gaussian. We find that the SDP-knockoff outperforms the BH method when the covariates are independent. However, we see that when the OLS test statistics are positively correlated, the TPRs of the oracle knockoff methods are close to zero. Table~\ref{table:main} shows the FDR and TPR of all methods for targeted FDR level $\alpha = 0.1$ and $0.2$. We find that the TPR of SDP-knockoff* is smaller than 0.02 when controlling the FDR at 0.2. 

We can understand the performance of knockoffs in each case by plotting the sorted log-odds values $\eta_{(j)}$, as we do in  Figure \ref{fig:etak}. While the knockoff* method can achieve the sorting pictued in Figure \ref{fig:etak}, feasible knockoff methods cannot achieve a perfect ordering because $\beta_j/\sigma$ is unknown. Recall that when the covariance matrix of the test statistics has factor model structure, the whitening step of knockoff destroys virtually all the information, and the log-odds of observing small $p$-values in the inference cannot rise above $-\log\alpha$. This is again confirmed by Figure \ref{fig:etak}, which shows the average of the $s=30$ largest log-odds across different trials. The rest of the log-odds are zero since the corresponding coefficient $\beta_j$ is zero. In particular, we find that most of the log-odds are smaller than $-\log\alpha$ with $\alpha = 0.1$.  Thus by Proposition \ref{prop:bigokd}, the number of rejections of any knockoffs method must be small. 

\subsection{Knockoff for multivariate Gaussian statistics}
\label{sec:mvnknockoff}

In Section ~\ref{sec:knockoffswhitened}, we reinterpreted the knockoff method and generalized the fixed-X knockoff procedure to multivariate normal test statistics. Here we use simulations to investigate the performance of different methods for testing the means of multivariate normal, and hint at the possible use cases and limitations of knockoffs for such problems.

We generate $d=1000$ dimensional multivariate Gaussian vectors $\htheta \sim \cN_d(\theta, K)$. The mean vector $\theta$ is generated in the same way as the linear coefficient $\beta$ in the previous simulation, except that the non-zero means $\theta_j$ are set to 3.5. We consider two types of covariance matrices:
\begin{enumerate}[(a)]
    \item \textbf{$K$ has factor model structure}. In particular, we let
    \begin{equation}
    \label{eq:Kfactormodel}
            K = I_d + \lambda\sum_{\ell=1}^{k}u_\ell u_\ell^T,
    \end{equation}
    where $u_\ell$ are drawn independently from the uniform sphere.
    \item \textbf{$K^{-1}$ has factor model structure}, i.e.
    \begin{equation}
    \label{eq:Kinvfactormodel}
            K^{-1} = I_d + \lambda\sum_{\ell=1}^{k}u_\ell u_\ell^T,
    \end{equation}
\end{enumerate}
In particular, we choose $k=2$ and $\lambda \in \{20, 100\}$. We normalize the covariance matrix $K$ such that $K$ has unit diagonal entries. After normalization, the largest eigenvalue of $K$ is approximately 19 and 75 for $\lambda=20$ and 100, respectively (the sum of all eigenvalue is $d=1000$). For each setup, we consider one fixed realization of the random matrix $K$. In each trial of our simulation, we first generate the mean vector $\theta$ and then the multivariate normal vector $\htheta$.

For each matrix $K$, we repeat the above data generating procedure $N=600$ times. 
Again, we consider both the SDP knockoff and the equi-correlated knockoff to create the diagonal matrix $\Delta$ that satisfies $\Delta\succeq K$. For each diagonal matrix $\Delta$, we first implement the oracular ordering and the associated knockoff* procedure defined in Section \ref{sec:knockoffstar}. 
For each observed $\htheta \sim \cN(\theta, K)$, we generate a pseudo-design matrix $X^*$ and pseudo-response $y^*$ as in \eqref{eq:packageknockoffs}.

Table \ref{table:mvnknockoff} shows the FDR and TPR of different methods for $\alpha = 0.2$. As expected, we found that the power of even the best achieveable knockoff method is less than that of Bonferroni when the covariance matrix $K$ has a factor model structure with reasonably large leading eigenvalue. This suggests that the knockoff-type approaches suffer from severe power loss when applied to general test statistics with factor model structure. However, when the precision matrix $K^{-1}$ has a factor model structure, it may be possible to use the knockoff framework procedure to design a test with superior TPR than baseline methods such as the BH. We leave this for future research.

\begin{figure}
    \centering
    \resizebox{\textwidth}{!}{
    \begin{tabular}{cc}
    \includegraphics[width=0.45\textwidth]{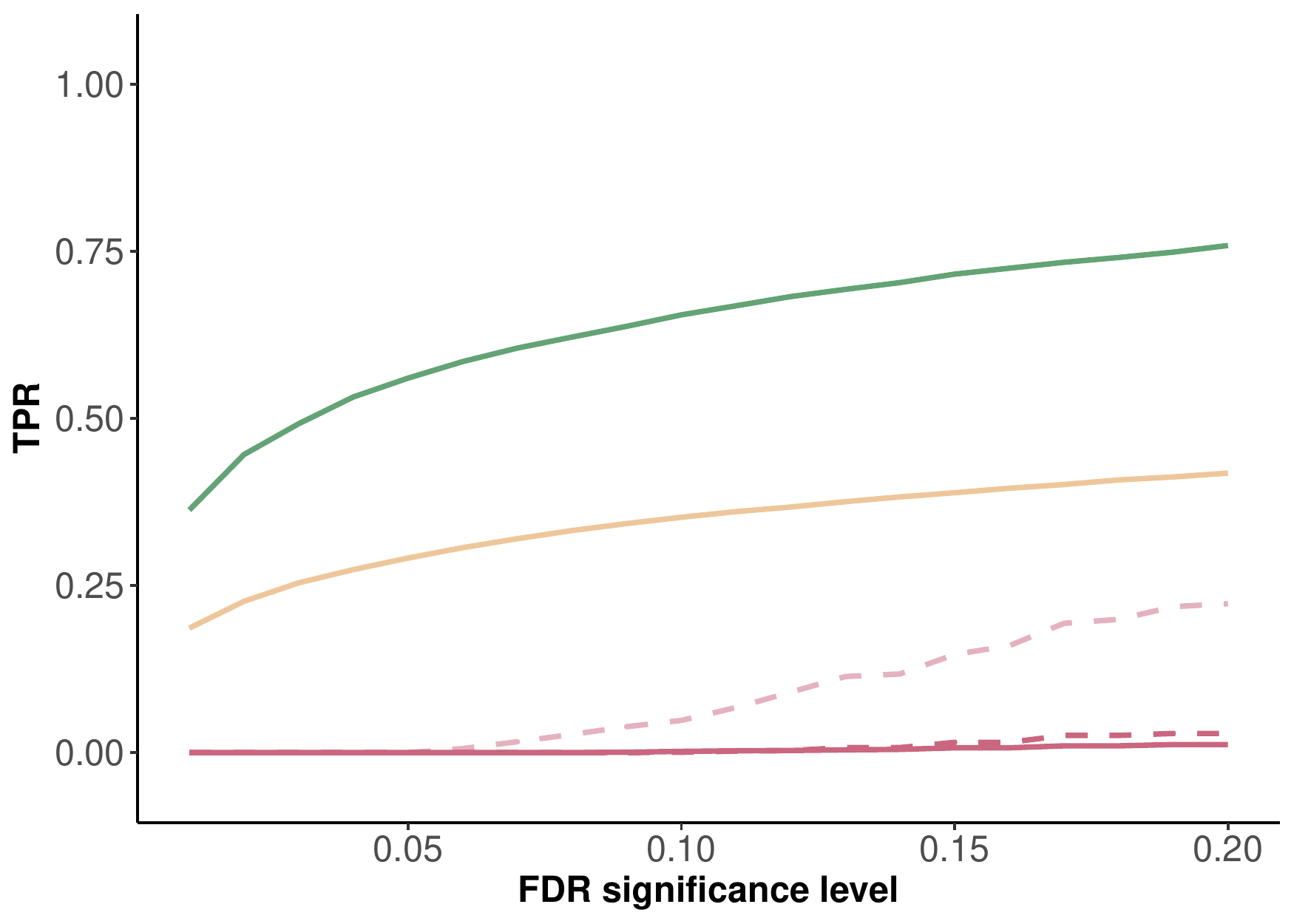} &
    \includegraphics[width=0.45\textwidth]{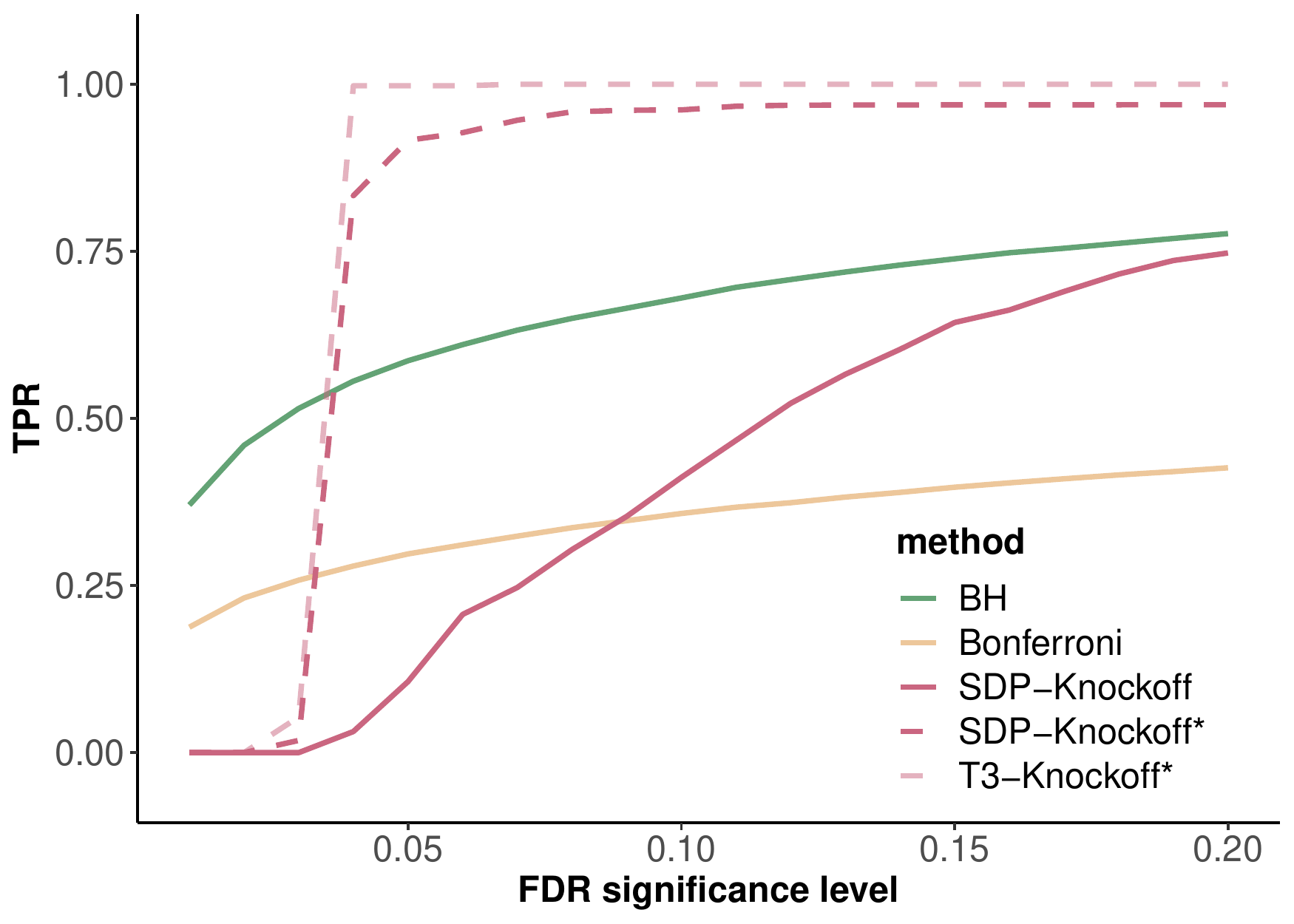} \\
         (a) Positively equi-correlated OLS estimator  &
         (b) Positively equi-correlated covariates       
    \end{tabular}
    }
    \begin{tabular}{c}
    \includegraphics[width=0.45\textwidth]{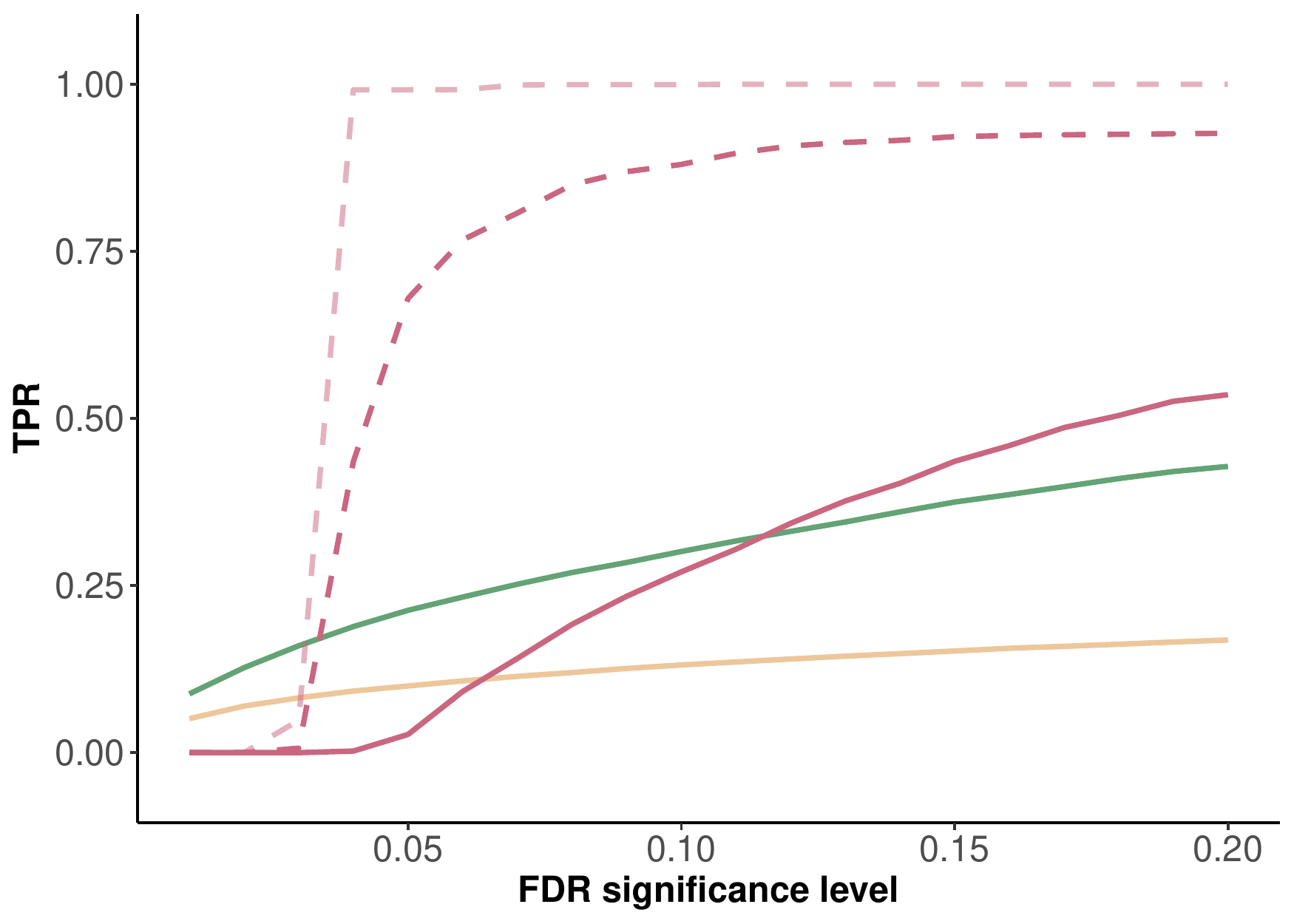} \\
        Reference: i.i.d Gaussian covariates with non-null $\beta_j = 3.5$.
    \end{tabular}
    \vspace{2mm}
    \caption{TPR of different tests under different FDR significance levels. Given the SDP-knockoff matrix, the SDP-knockoff* test is the best achievable knockoff method and can only be implemented with oracle knowledge about $\beta$. SDP-knockoff is a practically feasible knockoff method which uses the maximum lasso penalty level as the test statistics. The T3-knockoff* method defined in Algorithm~\ref{alg:t3knockoff}, and provides an upper bound on any knockoff methods where the analyst is not allowed to know the nonzero indices of $\beta$ at the time of determining the knockoff matrix $\tX$. In particular, the T3-knockoff* is more optimistic than the oracle procedure SDP-knockoff*.
    In Figure (a), the covariates in the design matrix are positively correlated with correlation approximately 0.2. In Figure (b), the covariates in the design matrix are negatively correlated, and the entries of the OLS estimate $\hbeta$ are positively correlated with correlation approximately 0.2. For reference, we also show the case where the covariates in the design matrix are i.i.d Gaussian, and with the non-null $\beta=3.5$, a regime considered by \citet{barber2015controlling}.}
    \label{fig:knockoff_power_example}
\end{figure}

\begin{figure}
    \centering
    \resizebox{\textwidth}{!}{
    \begin{tabular}{cc}
    \includegraphics[width=0.45\textwidth]{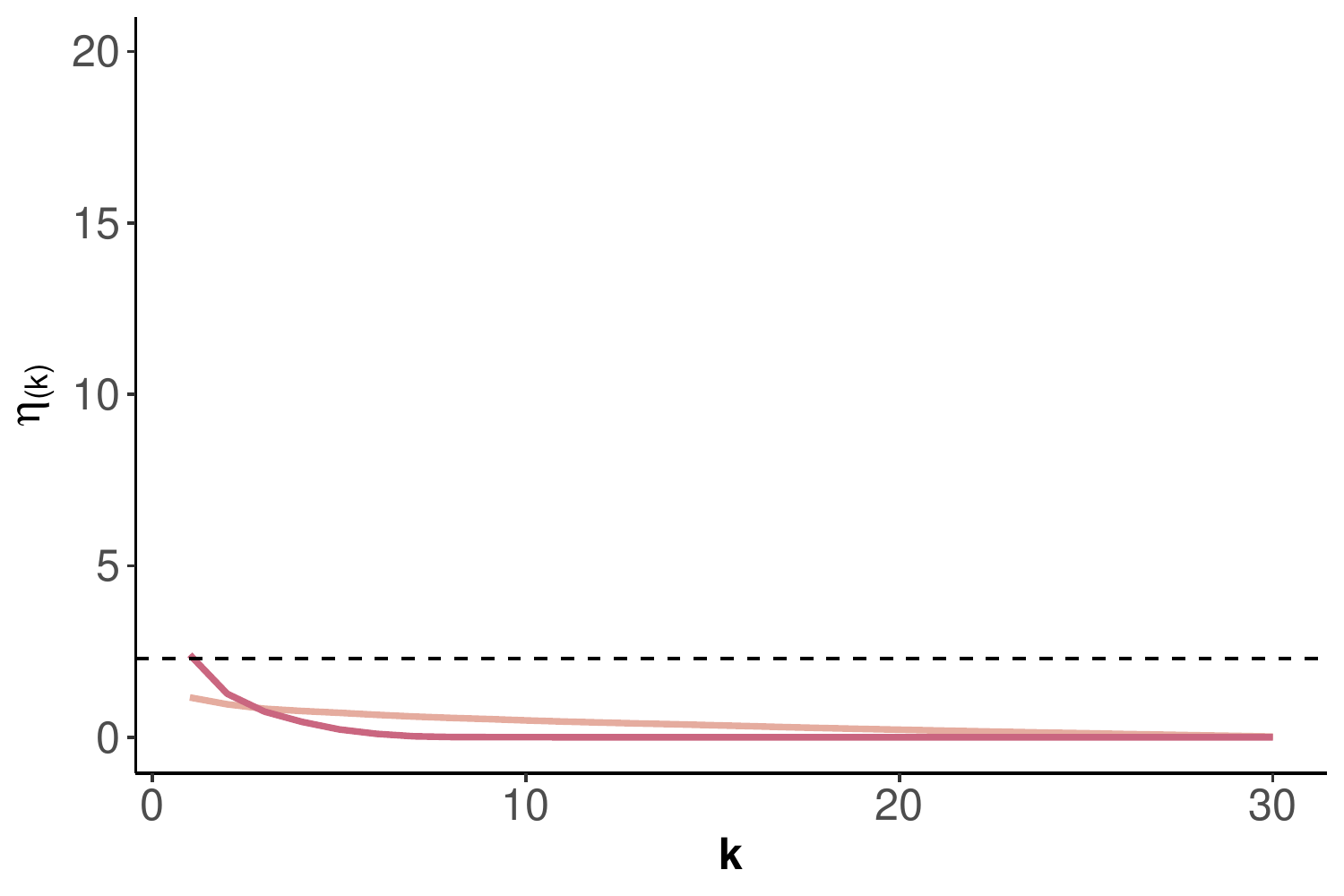} &
    \includegraphics[width=0.45\textwidth]{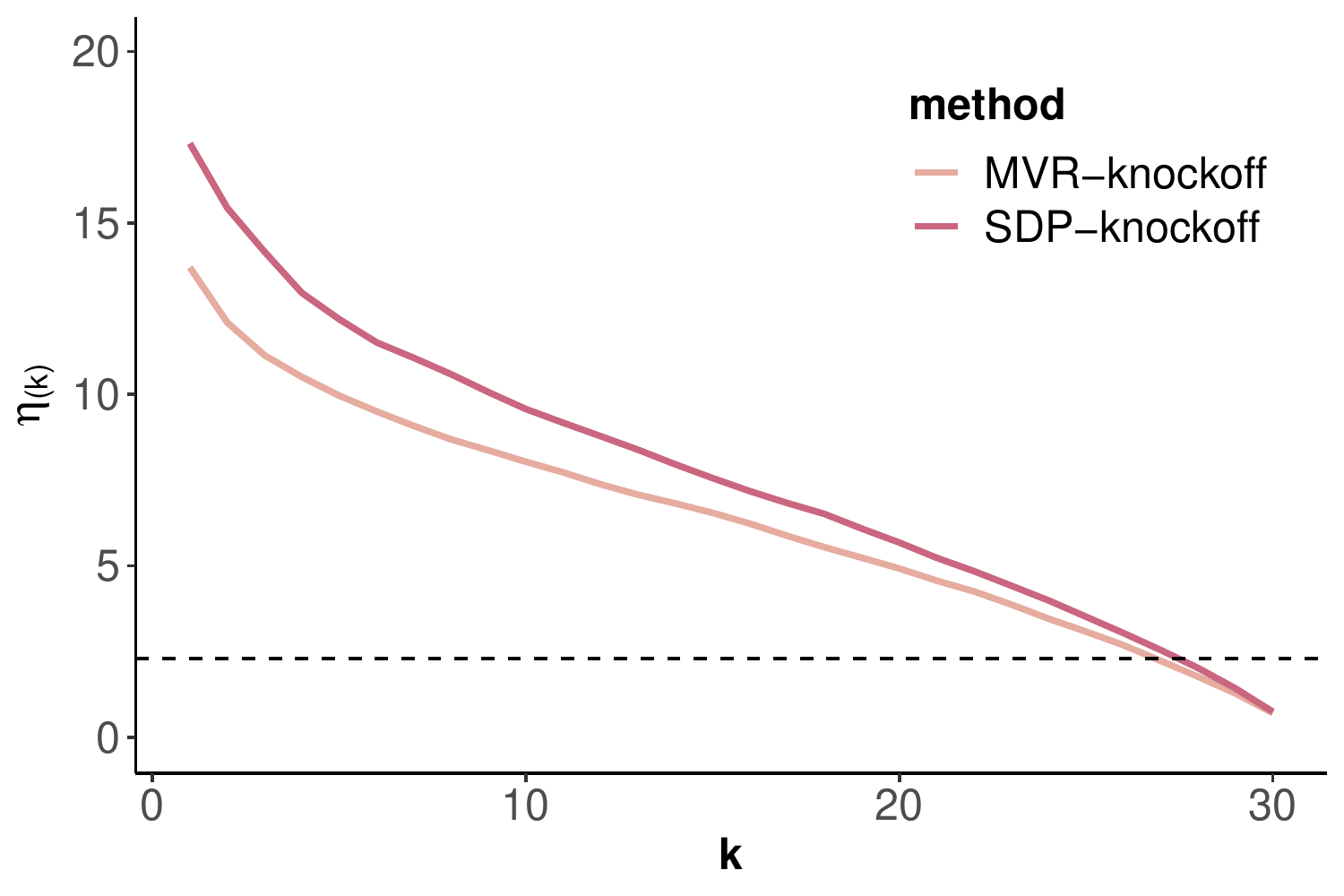} \\
         (a) Positively equi-correlated OLS estimator &
         (b) Positively equi-correlated covariates     
    \end{tabular}
    }
    \caption{The log-odds of observing a small p-value in the inference stage. The log-odds are computed as Equation~\ref{eq:etaj}. Only the $s=30$ non-zero log-odds are shown in both figures.  $\eta_{(k)}$ denotes the $k$th largest log-odd and $\eta_{(k)} = -\log 0.1$ is shown in the dashed black line.}
    \label{fig:etak}
\end{figure}

\begin{table*}\centering
\begin{subtable}[t]{\textwidth}
\centering
\resizebox{\textwidth}{!}{
\begin{tabular}{@{}ccccccccc@{}}\toprule
&& \multicolumn{2}{c}{SDP-knockoff} &\multicolumn{2}{c}{MVR-knockoff}  & \multicolumn{2}{c}{Other methods}\\
\cmidrule{3-4} \cmidrule{5-6} \cmidrule{7-8}
&& Knockoff* & Maximum penalty level & Knockoff* & Maximum penalty level & BH & Bonferroni \\ \midrule
\multirow{2}{*}{$\alpha = 0.1$} & FDR &0.00&0.00&0.00&0.00&0.09&0.01 \\
& TPR &0.00&0.00&0.00&0.00&0.65&0.35 \\ \midrule
\multirow{2}{*}{$\alpha = 0.2$} & FDR &0.00&0.01&0.00&0.00&0.18&0.02 \\
& TPR &0.03&0.00&0.02&0.01&0.76&0.42 \\
\bottomrule
\end{tabular}}
\caption{Positively equi-correlated OLS estimator}
\end{subtable}
\begin{subtable}[t]{\textwidth}
\centering
\resizebox{\textwidth}{!}{
\begin{tabular}{@{}cccccccc@{}}\toprule
&& \multicolumn{2}{c}{SDP-knockoff} &\multicolumn{2}{c}{MVR-knockoff} & \multicolumn{2}{c}{Other methods}\\
\cmidrule{3-4} \cmidrule{5-6} \cmidrule{7-8}
&& Knockoff* & Maximum penalty level & Knockoff* & Maximum penalty level & BH & Bonferroni \\ \midrule
\multirow{2}{*}{$\alpha = 0.1$} & FDR &0.06&0.07&0.05&0.06&0.10&0.01 \\
& TPR &0.96&0.44&0.95&0.44&0.68&0.36 \\ \midrule
\multirow{2}{*}{$\alpha = 0.2$} & FDR &0.16&0.17&0.15&0.16&0.19&0.01 \\
& TPR &0.97&0.77&0.96&0.74&0.78&0.43 \\
\bottomrule
\end{tabular}}
\caption{Positively equi-correlated covariates}
\end{subtable}
\caption{FDR and TPR of different methods under different target FDR levels.}
\label{table:main}
\end{table*}

\begin{table}\centering
\begin{subtable}[t]{\textwidth}
\centering
\resizebox{\textwidth}{!}{
\begin{tabular}{@{}cccccccc@{}}\toprule
&&\multicolumn{2}{c}{SDP-knockoff} & \multicolumn{2}{c}{MVR-knockoff} & \multicolumn{2}{c}{Other methods}\\
\cmidrule{3-4} \cmidrule{5-6} \cmidrule{7-8}
&& Knockoff* & Maximum penalty level & Knockoff* & Maximum penalty level  & BH & Bonferroni \\ \midrule
\multirow{2}{*}{$\lambda = 20$} & FDR &0.00&0.04&0.03&0.05&0.19&0.01 \\
& TPR &0.54&0.18&0.70&0.36&0.77&0.42 \\ \midrule
\multirow{2}{*}{$\lambda = 100$} & FDR &0.00&0.02&0.00&0.00&0.18&0.01 \\
& TPR &0.21&0.05&0.23&0.07&0.77&0.41 \\
\bottomrule
\end{tabular}
}
\caption{When $K$ has factor model structure}
\end{subtable}
\begin{subtable}[t]{\textwidth}
\centering
\resizebox{\textwidth}{!}{
\begin{tabular}{@{}cccccccc@{}}\toprule
&&\multicolumn{2}{c}{SDP-knockoff} & \multicolumn{2}{c}{MVR-knockoff} & \multicolumn{2}{c}{Other methods}\\
\cmidrule{3-4} \cmidrule{5-6} \cmidrule{7-8}
&& Knockoff* & Maximum penalty level & Knockoff* & Maximum penalty level & BH & Bonferroni \\ \midrule
\multirow{2}{*}{$\lambda = 20$} & FDR &0.18&0.20&0.18&0.19&0.19&0.02 \\
& TPR &1.00&0.71&0.99&0.70&0.76&0.41 \\ \midrule
\multirow{2}{*}{$\lambda = 100$} 
& FDR &0.18&0.15&0.19&0.17&0.19&0.02 \\
& TPR &0.99&0.45&0.99&0.48&0.77&0.41 \\
\bottomrule
\end{tabular}}
\caption{When $K^{-1}$ has factor model structure}
\end{subtable}
\caption{FDR and TPR of different methods for testing means of multivariate Gaussian with correlation matrix $K$ defined in Equations~\ref{eq:Kfactormodel} and \ref{eq:Kinvfactormodel}. Target FDR level $\alpha = 0.2$. ``Maximum Penalty Level" refers to the method where we first generate artificial design matrix and response $(X, y)$, and then apply the fixed-X knockoff.}
\label{table:mvnknockoff}
\end{table}

\section{Discussion}\label{sec:discussion}

\subsection{Knockoffs, randomized responses, and the selection-inference tradeoff}

Viewing knockoffs as a conditional post-selection inference method sheds light on what \citet{fithian2014optimal} called the {\em selection-inference tradeoff}: that the more we condition on, the less data remains for confirmatory analysis. This tradeoff is most obviously apparent in the case of data splitting, where the analyst selects a model or hypotheses to test by means of some exploratory analysis using a fraction of the data points, and then carries out confirmatory inference using only the remaining fraction, i.e. inference conditions on the initial data set (this assumes the two data sets are independent; otherwise data splitting may be invalid). A similar phenomenon is present in other conditional post-selection inference problems, where whatever statistics of the data we observe in the selection (exploratory) stage are unavailable as inferential evidence in the conditional inference (confirmatory) stage.

Compared to most other conditional inference methods, knockoffs conditions on much more about the data, holding out only the randomized and binarized $\sgn(\tbeta)$ for confirmatory inference. What is more, because $\sgn(\tbeta_j)\sim\text{Unif}\{-1,+1\}$ under $H_j$, the binary conditional $p$-values $\tp_j$ can never be smaller than $1/2$. The reason knockoff methods are nevertheless able to compete with and sometimes outperform other state-of-the-art multiple testing methods is because the FDR is an aggregate error criterion: to control it, knockoffs need never be confident about rejecting any individual hypothesis, only about the fraction of nulls early in the list. As a result, no individual $\tp_j$ needs to be minuscule, so long as most of the highly prioritized ones are $1/2$.

By giving up on making each $\tp_j$ powerful, knockoffs is able to use nearly all of the information in $\hbeta$ to supercharge the more flexible exploratory stage, betting on its ability to pack the front of the priority list with non-null hypotheses. This strategy of ``betting on exploration'' can pay off especially handsomely when Bayesian priors or structural assumptions like sparsity can be brought to bear during the exploration, which can use them in an unfettered way.

The whiteout phenomenon we describe here is an example of where that bet goes wrong, leaving too little information for inference. The fundamental problem is that the inference engine, Selective SeqStep requires {\em independent} binary $p$-values, which can only be created by adding enough noise to make $\Var(\tbeta)$ diagonal. If $\Sigma$ is ``too far from diagonal'' in the sense we describe, then this cannot be done without destroying the signal.

It may seem counterintuitive that adding more noise (larger $\Delta$) means using up more data for selection and leaving less for inference. In most methods that use randomized data for exploration, such as \citet{tian2018selective}, the opposite is true: adding more noise hides more information from the selection algorithm, preserving it for confirmatory inference. The difference is that, in knockoffs, the information ``left over'' after randomizing $\tbeta$ is also given to the analyst at exploration time, in the form of $\xi$. Instead it is the randomized $\tbeta$ that is (partly) held out for inference. 

We could equivalently define the whitening method in terms of Gaussian noise 
\[
\nu \;=\; -\Delta^{-1}\omega \;\sim\; \cN_d\left(0, \sigma^2 \Delta^{-1}\right), \quad \text{ with }\;\; \xi \;=\; A\hbeta + \nu
\]
viewed as a noisy version of $A\hbeta$, and $\tbeta = \Delta \Sigma^{-1} \hbeta - \Delta\xi$ viewed as the residual information. From this perspective, the ``noise variance'' is $\Delta^{-1}$, so that more noise (smaller $\Delta$) once again means holding out more information in the form of $\tbeta$.

\subsection{The whitening interpretation and \citet{spector2020powerful}}\label{sec:spectorjanson}

Our whitening interpretation sheds potentially interesting light on the phenomenon recently discovered by \citet{spector2020powerful} when $X^\tran X = \Sigma^{-1}$ is an equicorrelated covariance matrix with diagonal entries equal to $1$ and off-diagonal entries equal to $\rho \geq 0.5$. In this example the maximum eigenvalue of $\Sigma$ is $1$, so there is no ``whiteout'' problem, but the authors find that both equicorrelated and SDP knockoffs struggle to make any rejections. To understand why, note that both methods would set $D = 2(1-\rho)I_d$, so $\Delta = (1-\rho)^{-1}I_d$ and $A = \rho \mathbf{1}_d\mathbf{1}_d^\tran$. As a result, in the exploratory analysis the analyst observes 
\[
\xi \;\sim\; 
\cN\left( \rho \sum_{j=1}^d \beta_j, \; \rho\sigma^2\right) \cdot \mathbf{1}_d, \quad \text{ and }\; |\tbeta_j| \;\simind\; \frac{1}{1-\rho}\,|\cN(\beta_j, \sigma^2)|.
\]

Because only $\sum_j \beta_j$ and $|\beta|$ are identifiable from the exploratory data set, the analyst has no way to make an educated guess $\psi_j$ about $\sgn(\beta_j)$. This problem can be resolved by choosing $D$, or $\Delta$, more judiciously, as \citet{spector2020powerful} show.

\subsection{Power analysis in knockoffs}\label{sec:poweranalysis}

One takeaway message of Section~\ref{sec:bounds} is the crucial role played by the threshold $-\log\alpha$ in knockoffs methods' performance. If 
\begin{equation}\label{eq:badmu}
\mu_j \;=\; \frac{2\beta_j^2}{\sigma^2\Delta_{jj}} \;<\; -\log \alpha,
\end{equation}
then  we are more likely than not to observe $\eta_j < - \log \alpha$, since $\eta_j \sim \cN(\mu_j, 2\mu_j)$. In that case, even if we guess the alternative direction $\psi_j$ right we will still have
\begin{equation}\label{eq:badW}
\PP(W_j < 0) \;>\; \alpha \PP(W_j > 0), \quad\Longleftrightarrow\quad \PP(\tp_j = 1) \;>\; \alpha \PP(\tp_j = 1/2),
\end{equation}
so variable $j$ will be a net drag on the FDP estimator's struggle to remain below $\alpha$.

This observation can help us to do rudimentary power analysis at various stages of the procedure. For example, before we observe anything about the response vector $y$, we can inspect the diagonal entries of the matrix $\Delta = 2D^{-1}$ and ask how large $\beta_j/\sigma$ would have to be for us to have reasonable power to detect variable $j$. Doing a little algebra on inequality~\eqref{eq:badmu}, we arrive at
\begin{equation}\label{eq:SNRstar}
\mu_j \;=\; \frac{2\beta_j^2}{\sigma^2\Delta_{jj}} \;<\; -\log\alpha \quad \Longleftrightarrow \quad |\beta_j/\sigma| \;<\; \sqrt{\frac{-\Delta_{jj}\log\alpha}{2}} \;=\; \sqrt{\frac{-\log\alpha}{D_{jj}}}.
\end{equation}
We can think of \eqref{eq:SNRstar} as giving a critical threshold for the SNR of variable $j$. For example, suppose $\alpha = 0.05$, so $-\log\alpha \approx 3$. Then if $\Delta_{jj} = 6$ (or equivalently $D_{jj} = 1/3$), $|\beta_j/\sigma|$ should be larger than $3$ if we want $\eta_j$ to be above $-\log \alpha$ most of the time. Likewise, if $\Delta_{jj} = 32/3 \leq 11.7$ ($D_{jj} = 3/16 = 0.1875$), then the critical SNR threshold is about 4 for $\alpha = 0.05$. 

Importantly, because this variable-by-variable power analysis can be done {\em before} we observe anything about $y$, we can still change course if we don't like the $\Delta_{jj}$ values we get --- we could either choose a different knockoff matrix or abandon the knockoffs framework and use BH instead, without any threat to either method's FDR control guarantees.

Whereas our theoretical results emphasize lower bounds on $\Delta$, for an analyst intending to use knockoffs the more interesting question is how large each $\Delta_{jj}$ actually is in the specific knockoff matrix they are about to use for their problem. Until more is understood about what regimes lead knockoffs to dominate BH or vice versa, we recommend that analysts at least inspect the $\Delta_{jj}$ values in light of these SNR thresholds as a diagnostic tool. If, say, only a few of the $\Delta_{jj}$ values are below $10$, then the matrix $\Sigma$ may not be a suitable problem structure for knockoffs. In our stock market example, only two $\Delta_{jj}$ values are below $6$ and only $50$ are below $11.7$, suggesting that only very strong signals have a good chance of generating rejections.

\subsection{Concluding remarks}

We emphasize once again that the results we derive for these asymptotic regimes do not imply that fixed-$X$ knockoffs are underpowered as a general rule. On the contrary, we believe our results are interesting precisely because the opposite is true: there are many problems where existing fixed-$X$ knockoff methods outperform all other known FDR-controlling methods. In particular, BH and knockoffs represent two completely different approaches to multiple testing in regression or with multivariate normal test statistics. To give practitioners appropriate guidance about which one to use, more work is needed to answer several crucial questions: When do knockoff methods outperform the BH procedure, and which implementations perform the best? How can practitioners recognize which is better for their context? Can hybrid methods such as those of \citet{sarkar2021adjusting}, or methods yet to be developed, balance the tradeoffs between the two approaches, preserving the strengths of the knockoffs framework without suffering its drawbacks? By identifying pitfalls for the knockoffs framework our results represent strides toward a more complete understanding of multiple testing in the linear model.

\section{Proofs}
\label{sec:proofs}
\subsection{Proof of Proposition \ref{prop:knockoffstaropt}}

\knockoffstaropt*

\begin{proof}
If $d_1 = 0$ there is nothing to prove, so assume $d_1 \geq 1$. Given $\xi$, $|\tbeta|$, and $\psi$, the conditional $p$-values are independent with
\[
\logit\,\PP(\tp_j = 1/2) \;=\; \eta_j \cdot \sgn(\psi_j\beta_j).
\]
If we hold the ordering fixed, the rejection set is stochastically increasing in each of the above log-odds, so the TPP is always made stochastically larger by setting $\psi_j = \sgn(\beta_j)$ whenever $\beta_j \neq 0$. We can therefore restrict our attention to the case where the log-odds for each variable is $\eta_j \geq 0$.

Next, define the indicator $E_k = 1\{\hFDP_k^\wh \leq \alpha\}$. If $\eta_{[j]} < \eta_{[j+1]}$, then swapping the two leaves $E_k$ fixed for all $k \neq j$, but increases the conditional probability that $E_j = 1$ given $E_{-j}$. Therefore, $\hk$ is stochastically largest when the log-odds are arranged in decreasing order. The same is true for the number of rejections $R_{\hk} = \left\lceil \frac{1+\hk}{1+\alpha} \right\rceil$.

Likewise, if $\eta_{[j]} < \eta_{[j+1]}$ then $H_{[j+1]}$ must be non-null, so conditional on $E_{-j}$ the number of rejected non-null hypotheses is also made stochastically larger by arranging the log-odds in decreasing order.
\end{proof}

\subsection{Proof of Proposition \ref{prop:bigokd}}

For any $\alpha$ and $\delta$, we define
\[
p = \frac{\alpha}{1+\alpha}, \quad q_{\delta} = \frac{\alpha + \delta}{1 + \alpha + \delta}.
\]
We will now prove the following proposition, which is stronger and more precise than Proposition \ref{prop:bigokd}.

\begin{proposition}
\label{prop:bigokd_technical}
    Suppose that
    \[
     \eta_{(k^*)} < -\log (\alpha + \delta), \delta > 0,
    \]
    where $k^*\geq 1/\alpha$. Define
       \[
   C_1(\alpha, \delta) = \frac{1}{1+\alpha}\left[\max\left\{1, \frac{4\alpha(1+\alpha+\delta)}{\delta}\right\} + 1\right],
   \]
   and
   \[
   C_2(\alpha, \delta) = \frac{1}{1+\alpha}\left[\frac{e}{4\sqrt{\pi}}\frac{\sqrt{1+\alpha}q_\delta}{\sqrt{p(1-p)}}c_h^{-3/2}\min\left\{\frac{2(q_{\delta}-p)}{p(1-q_{\delta})}, \;1\right\} + 2 \right], 
\]
where
\[
\lambda^* = p + \frac{q_{\delta} - p}{4} < q_\delta, \quad \text{and} \quad c_h= - \left[\lambda^*\log \frac{q_\delta}{\lambda^*} + (1 - \lambda^*)\log \frac{1 - q_\delta}{1 - \lambda^*}\right] = O(\delta^2).
\]
    Then the expected number of rejections for any knockoff procedure at FDR significance level $\alpha$ is upper bounded by $C_1(\alpha, \delta)k^* + C_2(\alpha, \delta)$.

\end{proposition}

\begin{proof}

Consider the random walk
\[
    S_k = \sum_{j=1}^k (p - Z_j), \quad \text{where } Z_j \simind \text{Bern}\left(q_j\right), \quad p = \frac{\alpha}{1+\alpha}, \quad
    q_j = 
    \frac{e^{-\eta_{(j)}}}{1 + e^{-\eta_{(j)}}}.
\]
Then by \eqref{eq:Rformula}, for any knockoff procedure, the number of rejections $R$ is upper bounded by $\frac{2 + \hk}{1+\alpha}$, where
\[
\hk = \max\left\{k: S_k \geq \frac{1}{1+\alpha}\right\} \leq \max\left\{k: S_k \geq 0\right\}.
\]
Consider another random walk
\[
\tilde{S}_k = \sum_{j=1}^k (p - \tilde{Z}_j), \quad \text{where } \tilde{Z}_j \simiid \text{Bern}\left(q_{\delta}\right), \quad 
q_{\delta} = \frac{\alpha + \delta}{1 + \alpha + \delta} \leq q_{(k^*)}
\]
Note that
\[
S_k \;=\; \sum_{j=1}^{k^*} (p - Z_j) + \sum_{j=k^*}^d (p - Z_j) \;\leq\; pk^* + \sum_{j=k^*}^d (p - Z_j).
\]

Since $q_1\leq q_2\leq ...$, we know that
$\sum_{j=k^*}^d (p - Z_j)$ is stochastically smaller than $\sum_{j=k^*}^d (p - \tilde{Z}_j)$, where $\tilde{Z}_j \simiid \text{Bern}\left(q_{(k^*)}\right)$.
Therefore $S_k$ is stochastically smaller than $pk^* + \tilde{S}_{k-k^*}.$
Therefore
\begin{equation}
\label{eq:rejectionnumberbound}
    \EE[\hk] \;\leq\; \EE[\max\left\{k: S_k \geq 0\right\}] \;\leq\; \EE\left[\max\left\{k: \tilde{S}_{k-k^*} \geq -pk^*\right\}\right] \;=\; k^* + \EE\left[\max\left\{k: \tilde{S}_{k} \;\geq\; -pk^*\right\}\right].
\end{equation}
Define
\[
p(r) \stackrel{\Delta}{=} \PP\left(\max\left\{k: \tilde{S}_k \geq -pk^*\right\} \;=\; r\right).
\]
Let $C_m = \max\{4\left(\frac{q_{\delta}}{p}-1\right)^{-1}, 1\}$, and then
\begin{equation}
\begin{aligned}
\label{eq:lasthittingtime}
&\EE\left[\max\left\{k: \tilde{S}_{k} \geq -pk^*\right\}\right] \\
\;=\; &\sum_{r=1}^{\infty}r\PP\left(\max\left\{k: \tilde{S}_{k} \geq -pk^*\right\}=r\right) \\
\;=\; &\sum_{r=1}^{C_mk^*}r\PP\left(\max\left\{k: \tilde{S}_{k}
-pk^*\right\}=r\right) +
\sum_{r=C_mk^* + 1}^{\infty}r\PP\left(\max\left\{k: \tilde{S}_{k}
-pk^*\right\}=r\right)\\
\;\leq\; &C_mk^*\sum_{r=1}^{C_mk^*}\PP\left(\max\left\{k: \tilde{S}_{k}
-pk^*\right\}=r\right) +
\sum_{r=C_mk^* + 1}^{\infty}r\PP\left(\max\left\{k: \tilde{S}_{k}
-pk^*\right\}=r\right)\\
\;\leq\; &C_mk^* + \sum_{r=C_mk^* + 1}^{\infty}r\PP\left(\max\left\{k: \tilde{S}_{k}
-pk^*\right\}=r\right) \\
\;=\;& C_mk^* + \sum_{r=C_mk^* + 1}^{\infty}rp(r).
\end{aligned}
\end{equation}
Therefore, combining Equations \ref{eq:lasthittingtime} and \ref{eq:rejectionnumberbound}, we have
\[
    \EE[\hk] \;\leq\; k^* + \EE\left[\max\left\{k: \tilde{S}_{k} \geq -pk^*\right\}\right] \;\leq\; k^* + C_mk^* + \sum_{r=C_mk^* + 1}^{\infty}rp(r) = (C_m+1)k^* + \sum_{r=C_mk^* + 1}^{\infty}rp(r).
\]
Recalling the definitions of $C_m$ and $q_{(k^*)}$, we have
\[
4\left(\frac{q_{\delta}}{p}-1\right)^{-1} 
\;=\; 4\left( \frac{(\alpha + \delta)(1 + \alpha)}{1 + \alpha + \delta} -1 \right)^{-1}
\;=\; \frac{4\alpha(1+\alpha+\delta)}{\delta}.
\]
Therefore we have $C_m + 1 \leq C_1(\alpha, \delta)$. 
Turning to the second term, by Lemma \ref{lemma:pr} we have
\begin{equation}
    \sum_{r=C_mk^* + 1}^{\infty}rp(r) \leq \frac{q_\delta e }{2\pi \sqrt{p(1-p)}}  \,\cdot\, \min\left\{\frac{2(q_{\delta}-p)}{p(1-q_{\delta})}, \;1\right\}
    \sum_{r=C_mk^* + 1}^{\infty} \sqrt{r}e^{-c_hr}.
\end{equation}
Note that we can bound the summation $\sum_{r=C_mk^* + 1}^{\infty} \sqrt{r}e^{-c_hr}$ by (note that $C_mk^* + 1 \geq 1/\alpha + 1$)
\begin{equation}
    \sum_{r=1/\alpha + 1}^{\infty} \sqrt{r}e^{-c_hr} \;\leq\; \sqrt{1+\alpha}\sum_{r=1/\alpha + 1}^{\infty} \sqrt{r-1}e^{-c_hr} \;\leq\; \sqrt{1+\alpha}\int_{1/\alpha}^{\infty} \sqrt{r}e^{-c_hr}dr \;\leq\; \sqrt{1+\alpha}\int_{0}^{\infty} \sqrt{r}e^{-c_hr}dr,
\end{equation}
where
\[
\int_{0}^{\infty} \sqrt{r}e^{-c_hr}dr \;=\; \int_{0}^{\infty} 2(2c_h)^{-3/2}y^2e^{-\frac{y^2}{2}}dy \;=\;(2c_h)^{-3/2} \int_{-\infty}^{\infty} y^2e^{-\frac{y^2}{2}}dy \;=\; \sqrt{2\pi}(2c_h)^{-3/2}
\]
The first equality above is obtained by change of variable $r=y^2/2c_h$.
Therefore,
\[
\sum_{r=C_mk^* + 1}^{\infty} \sqrt{r}e^{-c_hr} \leq \sqrt{1+\alpha}\sqrt{2\pi}(2c_h)^{-3/2}.
\]
Thus
\[
\sum_{r=C_mk^* + 1}^{\infty}rp(r)  \;\leq\;
\frac{e}{4\sqrt{\pi}}\frac{\sqrt{1+\alpha}}{\sqrt{p(1-p)}}c_h^{-3/2}\min\left\{\frac{2(q_{\delta}-p)}{p(1-q_{\delta})}, \;1\right\}
\]
and the proposition is proved.
\end{proof}

\subsection{Proof of results in Sections~\ref{sec:boundpower}--\ref{sec:boundpowerrandom}}

\main*

\begin{proof}
Assume without loss of generality that $\Delta_{11} \geq \Delta_{22} \geq \cdots \geq \Delta_{dd}$, so we have for all $j > k$ that
\[
\Delta_{jj} \;\geq\; b_j(\Sigma) \;\geq\; jb_k(\Sigma)/k.
\]
If $\beta_{j}^2 \leq \beta_{(k)}^2$, then we have
\[
\mu_j \;=\; \frac{2\beta_j^2}{\sigma^2 \Delta_{jj}} \;\leq\; \frac{2k\beta_{(k)}^2}{\sigma^2 j b_k(\Sigma)} \;\leq\; \frac{-k\log\alpha}{2j}.
\]
Define 
\[
N_\alpha = \#\left\{i \geq 1:\; \beta_{k+i}^2 \leq \beta_{(k)}^2, \;\eta_{k+i} > -\frac{1}{2}\log \alpha\right\}.
\]
At most $k-1$ indices $j$ have $\beta_j^2 > \beta_{(k)}^2$, so the total number of indices with $\eta_j \geq -\frac{1}{2}\log\alpha$ is at most $2k + N_{\alpha} - 1$. If $N_\alpha = n$, then we have
\begin{equation}\label{eq:etabound:pf2}
\eta_{(2k+n)} \;\leq\; -\frac{1}{2}\log\alpha,
\end{equation}
in which case Proposition~\ref{prop:bigokd} implies that the conditional expected number of rejections is bounded above by $(2k+n)C_1^\dagger(\alpha) + C_2^\dagger(\alpha)$, where $C_i^\dagger(\alpha) = C_i(\alpha, \sqrt{\alpha} - \alpha)$ for $i=1,2$. Therefore the expected number of rejections is 
\begin{align*}
\EE R &\;=\; \sum_n \EE [R \mid N_\alpha = n] \, \PP(N_\alpha = n)\\
&\;\leq\; \sum_n \PP(N_\alpha = n) \cdot \left[(2k+n)C_1^\dagger(\alpha) + C_2^\dagger(\alpha)\right] \,\\
&\;\leq\; (2k+\EE N_\alpha)C_1^\dagger(\alpha) + C_2^\dagger(\alpha).
\end{align*}
Applying Lemma~\ref{lem:Nalpha} with $z_i = \eta_{k+i}$ implies $\EE N_\alpha \leq C_3(\alpha) k$, so the result holds for 
\[
    C_1^*(\alpha) \;=\; (2+C_3(\alpha))\,C_1(\alpha, \sqrt{\alpha}-\alpha), \quad\text{ and }\;\;
    C_2^*(\alpha) \;=\; C_2(\alpha, \sqrt{\alpha}-\alpha).
\]
If $\beta_{(k)}$ is replaced with $\beta_{(1)}$ in \eqref{eq:finitesamplecondition} then the same proof applies, except that instead of $k-1$ indices with large $\beta_j$, there are none. Then we can replace $2k+n$ with $k+1+n$ in \eqref{eq:etabound:pf2}, and the result holds for
\[
C_1^*(\alpha) \;=\; (1+C_3(\alpha))\,C_1(\alpha, \sqrt{\alpha}-\alpha), \quad\text{ and }\;\;
    C_2^*(\alpha) \;=\; C_1(\alpha, \sqrt{\alpha}-\alpha) + C_2(\alpha, \sqrt{\alpha}-\alpha).
\]

\end{proof}

\mainrandom*

\begin{proof}
Define $K = \lfloor k/\pi_1 \rfloor$ and assume without loss of generality that $\Delta_{11} \geq \Delta_{22} \geq \cdots \geq \Delta_{dd}$, so we have for all $j > K$ that
\[
\Delta_{jj} \;\geq\; b_j(\Sigma) \;\geq\; jb_K(\Sigma)/K.
\]
If $\beta_{j}^2 \leq \beta_{(k)}^2$, then we have for $j > K$,
\[
\mu_j \;=\; \frac{2\beta_j^2}{\sigma^2 \Delta_{jj}} \;\leq\; \frac{2K\beta_{(k)}^2}{\sigma^2 j b_K(\Sigma)} \;\leq\; \frac{-K\log\alpha}{2j}.
\]
Define 
\[
N_\alpha^{(1)} = \#\left\{j \leq K:\; \beta_j \neq 0\right\}, \quad \text{ and } \;\; 
N_\alpha^{(2)} = \#\left\{i \geq 1:\; \beta_{K+i}^2 \leq \beta_{(k)}^2, \;\eta_{K+i} > -\frac{1}{2}\log \alpha\right\}.
\]
At most $k-1$ indices $j$ have $\beta_j^2 > \beta_{(k)}^2$, so the total number of indices with $\eta_j \geq -\frac{1}{2}\log\alpha$ is at most $k + N_{\alpha}^{(1)} + N_{\alpha}^{(2)} - 1$. If $N_\alpha^{(1)} = n_1$ and $N_{\alpha}^{(2)} = n_2$, then we have
\begin{equation}\label{eq:etabound:pf3}
\eta_{(k+n_1+n_2)} \;\leq\; -\frac{1}{2}\log\alpha,
\end{equation}
in which case Proposition~\ref{prop:bigokd} implies that the conditional expected number of rejections is bounded above by $(2k+n)C_1^\dagger(\alpha) + C_2^\dagger(\alpha)$, where $C_i^\dagger(\alpha) = C_i(\alpha, \sqrt{\alpha} - \alpha)$ for $i=1,2$. Therefore the expected number of rejections is 
\begin{align*}
\EE R &\;\leq\; (k+\EE N_\alpha^{(1)}+\EE N_\alpha^{(2)})C_1^\dagger(\alpha) + C_2^\dagger(\alpha).
\end{align*}
Because $\PP\left(\beta_j \neq 0\right) = \pi_1$ for every $j$, we have $\EE N_\alpha^{(1)} = \pi_1 K \leq k$. For $j > K$, we can write 

Applying Lemma~\ref{lem:Nalpha} with $z_i = \eta_{k+i}$ implies $\EE N_\alpha \leq C_3(\alpha) k$, so the result holds for 
\[
    C_1^*(\alpha) \;=\; (2+C_3(\alpha))\,C_1(\alpha, \sqrt{\alpha}-\alpha), \quad\text{ and }\;\;
    C_2^*(\alpha) \;=\; C_2(\alpha, \sqrt{\alpha}-\alpha).
\]
If $\beta_{(k)}$ is replaced with $\beta_{(1)}$ in \eqref{eq:finitesamplecondition} then the same proof applies, except that instead of $k-1$ indices with large $\beta_j$, there are none. Then we can replace $2k+n$ with $k+1+n$ in \eqref{eq:etabound:pf2}, and the result holds for
\[
C_1^*(\alpha) \;=\; (1+C_3(\alpha))\,C_1(\alpha, \sqrt{\alpha}-\alpha), \quad\text{ and }\;\;
    C_2^*(\alpha) \;=\; C_1(\alpha, \sqrt{\alpha}-\alpha) + C_2(\alpha, \sqrt{\alpha}-\alpha).
\]

\end{proof}

\maximin*

\begin{proof}
    Let $\Pi$ denote a uniformly random permutation on $d$ elements, and let $\Pi(j) \in \{1,\ldots,d\}$ denote the index where $j$ is sent by $\Pi$. Because $\Sigma$ is constant, to limit notational bloat we will suppress the input $\Sigma$ in $b_k(\Sigma)$, writing $b_1 \geq b_2 \geq \cdots \geq b_d$ instead. We assume without loss of generality that $\Delta_{11} \geq \cdots \geq \Delta_{dd}$. Finally fix a representative in $\Omega$ as
    \[
    \beta^* = (|\beta_{(1)}|,\ldots,|\beta_{(d)}|).
    \]

    The proof of Theorem~\ref{thm:mainrandom} argues that $C_1^*(\alpha)k + C_2^*(\alpha)$ is an analytic upper bound for the expected number of total rejections when we apply the knockoff* procedure with $\eta_j$ sampled from
    \begin{equation}\label{eq:etaforalgo}
    \eta_j \;\simind\; |\cN(\mu_j^*, 2\mu_j^*)|, \quad \text{ with } \;\;\mu_j^* \;=\; \frac{2\beta_{\Pi(j)}^{*2}}{b_j} 
    \end{equation}
    As a result, $\frac{1}{d_1}B_T$ is an analytic upper bound for the $\TPR$ of the same process.
    
    By contrast, Algorithm~\ref{alg:t3knockoff} directly samples $\eta$ from the same distribution \eqref{eq:etaforalgo} and then directly simulates the knockoff* procedure with those $\eta$ values and in each Monte Carlo iteration directly calculates $\TPP \leq R/d_1$. As a result,  $B_T \geq \EE\,R/d_1 \geq \EE\, \TPP = \EE \,\widehat{\TPR}$. 
    
    Next, consider the behavior of any feasible knockoff method $\cR$. Because $\Omega$ is a finite set, its lowest power on any $\beta\in \Omega$ is well-defined, and is no larger than its average power under sampling from any distribution over $\Omega$:
    \[
    \EE_{\beta \sim P} \,\TPR(\cR, \beta) \;\geq\; \min_{\beta\in\Omega} \TPR(\cR,\beta),
    \]
    where $P$ is the distribution of $(\beta^*_{\Pi(1)},\ldots,\beta^*_{\Pi(d)})$. Conditional on $\beta$, $\eta_1,\ldots,\eta_d$ are then distributed as
    \[
    \eta_j \;\simind\; |\cN(\mu_j, 2\mu_j)|, \quad \text{ with } \;\;\mu_j \;=\; \frac{2\beta_{\Pi(j)}^{*2}}{\Delta_{jj}} \;\leq\; \mu_j^*,
    \]
    and once they are generated $\cR$ will choose the orderings and directions in some way that is suboptimal compared to knockoff*. As a result, $B_A \geq \EE_{\beta \sim P} \,\TPR(\cR, \beta)$ and we have the result.
\end{proof}

\randombeta*

\begin{proof}
We begin with the first claim. When the variance $\sigma^2$ is known, and the test statistics are $z$ statistics, the claim follows from the fact that the proportion of non-null $p$-values that are smaller than $\alpha/d$ converges to 1 when all the non-null $\beta_j = \sqrt{2r\log d}$ for some $r>1$. When the variance $\sigma^2$ is unknown, we have access to an $\chi^2_d$-distributed variance estimate $\hat\sigma_d^2 \toProb \sigma^2$.
Therefore, the proportion of non-null $t$-statistics that are larger than $\sqrt{(1 + r)\log d} < \sqrt{2r\log d}$ converges to one. On the other hand, since $\chi^2_d/d \toProb 1$, we have that the c.d.f. of $t_d$ distribution at $-\sqrt{(1 + r)\log d}$ is smaller than the gaussian c.d.f at $-\sqrt{2\log d}$, for $d$ succifiently large. This proves that the TPR also converges to 1 when we use OLS $t$-statistics.

To prove the second claim, we will first lower bound $b_{ad}(\Sigma)$ for a generic fraction $a > 0$ to show that most of the diagonal entries of $\Delta$ are prohibitively large. By assumption, there exists some constant $c > 0$ for which $F(c) < a/2$. Then for large enough $d$ we have $F_d(c) \leq a/2$ as well, so that $|u_{1,(\lfloor ad/2\rfloor)}| \geq c/\sqrt{d}$, and we have
\[
b_{\lceil ad\rceil}(\Sigma) \;\geq\; \lambda_1\sum_{j=1}^{\lceil ad\rceil} u_{1,(j)}^2 \;\geq\; \frac{\lambda_1 ad}{2}\, u_{1,(\lfloor ad/2\rfloor)}^2 \;\geq\; \frac{\lambda_1 ac^2}{2}.
\]
By assumption, $\lambda_1/\log d \to \infty$, so it follows that for sufficiently large $d$ we have
\[
\frac{2\beta_{(1)}^2}{\sigma^2 b_{\lceil ad\rceil}(\Sigma)} \;=\; \frac{2\beta_0^2}{\sigma^2 b_{\lceil ad\rceil}(\Sigma)} \;=\; \frac{4\log d}{b_{\lceil ad\rceil}(\Sigma)} \;<\; -\frac{1}{2}\log\alpha.
\]
As a result, by Theorem~\ref{thm:main}, the expected number of rejections is bounded above by $C_1^*(\alpha)(ad+1) + C_2^*(\alpha)$. If $\liminf d_1/d > 0$, this completes the proof, since $a$ and $\alpha$ were arbitrary.

If $\supp(\beta)$ is uniformly random, then by Theorem~\ref{thm:mainrandom}, the expected number of rejections is bounded above by $C_1^*(\alpha)(\pi_1 ad+1) + C_2^*(\alpha) = C_1^*(\alpha)(ad_1+1) + C_2^*(\alpha)$. Again, this completes the proof since $a$ and $\alpha$ were arbitrary.
\end{proof}

\subsection{Proof of Theorem \ref{thm:asymptoticnormal}}

\asymptotic*

\begin{proof}
    Because $\hSigma_n \toProb \Sigma$, it will be nonsingular with probability approaching 1; likewise for $\hDelta_n \succeq \hSigma_n$. Because $\Delta(\cdot)$ is continuous, we also have $\hDelta_n \toProb \Delta = \Delta(\Sigma)$ and $\hM_n \toProb (\Delta - \Sigma)^{1/2}$. Then
    \[
    \begin{pmatrix} \xi_n \\[5pt] \tbeta_n\end{pmatrix} \;=\; \begin{pmatrix} \hSigma_n^{-1} - \hDelta_n^{-1} & -\hDelta_n^{-1}\hM_n \\[5pt] I_d & \hM_n \end{pmatrix} \;\begin{pmatrix} \hbeta_n \\[5pt] \omega'\end{pmatrix}.
    \]
    Because all four blocks of the matrix are converging in probability, $\hbeta_n$ is converging in distribution by assumption, and $\omega'$ has a fixed distribution, 
    \[
    (\xi_n, \tbeta_n) \;\Rightarrow\; (\xi,\tbeta) \;=\; \left(\Sigma^{-1}\hbeta - \Delta^{-1}(\hbeta + \omega), \; \hbeta + \omega\right).
    \]
    Because $W^+(\cdot)$ and $W^-(\cdot)$ are both continuous, we also have 
    \[
    \left(W^+(\xi_n,\tbeta_n),\; W^-(\xi_n,\tbeta_n)\right) \;\Rightarrow\; \left(W^+(\xi,\tbeta),\; W^-(\xi,\tbeta)\right).
    \]
    To complete the proof, for $(\xi,\tbeta)\in E$, let $\Pi(\xi,\tbeta)$ denote the unique permutation on $2d$ elements that arranges all $2d$ variables in decreasing order. Then $E$ is the disjoint union of $E_\pi = E \cap \Pi^{-1}(\pi)$, where $\Pi^{-1}$ is the preimage and $\pi$ ranges over all permutations. Because there are only finitely many permutations and the rejection set $\cR$ is constant on each $E_\pi$, it remains only to show for each $\pi$ that
    \begin{equation}\label{eq:probpi}
    \PP\left( (\xi_n,\tbeta_n) \in E_\pi\right) \to \PP\left( (\xi,\tbeta) \in E_\pi\right).
    \end{equation}
    For example, if we take $\pi$ to be the identity permutation, then
    \[
    E_\pi \;=\; \{W_1^+ > \cdots > W_d^+ > W_1^- > \cdots > W_d^- > 0\}.
    \]
    Let $\delta_\pi(\xi,\tbeta)\in \RR$ denote the largest margin of any of the $2d$ inequalities defining $E_\pi$, possibly negative, so that $E_\pi = \{\delta_\pi > 0\}.$ Because $\delta_\pi$ is continuous in $(W^+,W^-)$ it too is converging in distribution; let $F_{\pi,n}$ and $F_{\pi}$ denote the cumulative distribution functions of $\delta_\pi(\xi_n,\tbeta_n)$ and $\delta_\pi(\xi,\tbeta)$, respectively. Because $\delta_\pi^{-1}(0) \subseteq E^\setcomp$ has Lebesgue measure zero, 0 is a continuity point of $F_\pi$, so we have \eqref{eq:probpi} and the proof is complete.
\end{proof}

\subsection{Proof of technical lemmas}

\Nalpha*

\begin{proof}
Define $c = -\frac{1}{2}\log\alpha > 0$ and $N_\alpha = \#\{i: |z_i| > c\}$.
\[
c = -\frac{1}{2}\log\alpha > 0, \quad t_i = i/k, \quad \text{ and } \; N_\alpha = \#\{i: |z_i| > c\}.
\]
Because $\PP(|z_i| > c)$ is increasing in $\mu_i$ for $\mu_i \in [0,c]$, we have 
\[
\EE N_\alpha \;=\; \sum_{i=1}^\infty \PP(|z_i| > c) \;\leq\; \sum_{i=1}^\infty f(i/k), \quad \text{ where } \; f(t) = \PP_{\mu=\frac{c}{1+t}}\left(|\cN(\mu,2\mu)| > c\right).
\]
Because $f$ is decreasing, we have
\[
\sum_{i=1}^\infty f(i/k) \;=\; \sum_{i=1}^\infty k\int_{(i-1)/k}^{i/k} f(i/k)\,dt \;\leq\; k\int_0^\infty f(t)\,dt.
\]
Writing $f$ explicitly, we obtain
\[
\EE N \;\leq\; k \int_0^\infty \left( \Phi\left( \frac{-(2+t)\sqrt{c}}{\sqrt{2+2t)}}\right) \,+\, 1 - \Phi\left(\frac{t\sqrt{c}}{\sqrt{2+2t}} \right) \right)\,dt.
\]
The integral is finite since $1-\Phi(x)= \Phi(-x) \leq \exp(-x^2/2)$ for large $x$. Evaluating the integral numerically gives $C_3(0.05) \leq 1.05$, $C_3(0.1) \leq 1.37$, and $C_3(0.2) \leq 2.02$.

\end{proof}

\begin{lemma}
\label{lemma:rwdrifttechnical1}
For any $q, \theta \in (0, 1)$, we have
\[
\frac{q\log(1-\theta)}{\log(1-q\theta)} - 1 \geq \frac{(1-q)}{2}\theta.
\]
\end{lemma}
\begin{proof}
For any $q\in (0, 1)$, let 
\[
f(\theta) = \frac{(1-q)}{2}\theta \log(1-q\theta) + \log(1-q\theta) - q\log(1-\theta).
\]
Because $\log(1-q\theta) < 0$, it suffices to show that $f(\theta)\geq 0$ for $\theta\in(0, 1)$. First, note that $f(0) = 0$. We will show next that $f'(\theta)>0$ when $\theta > 0$:
\begin{align*}
  f'(\theta) & = \frac{(1-q)}{2}\left( \log(1-q\theta) - \frac{q\theta}{1-q\theta}\right) - \frac{q}{1-q\theta} + \frac{q}{1-\theta}\\
  & = \frac{(1-q)}{2} \left(\log(1-q\theta) + \frac{q\theta(1+\theta)}{(1-\theta)(1-q\theta)}\right).
\end{align*}
Note that
\[
\log(1-q\theta) = - \log\frac{1}{(1-q\theta)} \geq -\left( \frac{1}{(1-q\theta)} - 1 \right) = \frac{-q\theta}{1-q\theta}.
\]
Therefore
\begin{align*}
  f'(\theta) 
  & = \frac{(1-q)}{2} \left(\log(1-q\theta) + \frac{q\theta(1+\theta)}{(1-\theta)(1-q\theta)}\right) \\
  & \geq \frac{(1-q)}{2} \left(\frac{-q\theta}{1-q\theta} + \frac{q\theta(1+\theta)}{(1-\theta)(1-q\theta)}\right) \\  
  & = \frac{q(1-q)\theta^2}{(1-\theta)(1-q\theta)}> 0,
\end{align*}
and the lemma is proved.
\end{proof}
\begin{lemma}
\label{lemma:rwdrift}
Consider a random walk $S_t = \sum_{i=1}^ {t}\zeta_t$ where $\zeta_t$ are i.i.d Bernoulli variables with $\PP(\zeta_t = 1) = 1 - q$ and $\PP(\zeta_t = -1/\alpha) = q$. Let $p = \alpha/(1+\alpha)$ and suppose that $q > p$. Then 
\[
\PP\left(\max_{t \geq 1} S_t \leq 0\right) \leq q\min\left\{\frac{2(q-p)}{p(1-q)}, \,1\right\}
\]
\end{lemma}
\begin{proof}
We begin by identifying the martingale associated with the moment generating function of $S_t$.
Let $\psi_0>0$ be a positive value that satisfies
\begin{equation}
\label{eq:MGFmartingale}
\EE e^{\psi_0\zeta_1} = (1-q)e^{\psi_0} + qe^{-\psi_0\frac{1}{\alpha}} = 1.
\end{equation}
We will prove later that such $\psi_0$ exists and is unique.
It follows that $e^{\psi_0S_t}$ is a martingale, since
\[
\EE [e^{\psi_0S_{t+1}}|S_t] = e^{\psi_0S_t}\EE e^{\psi_0\zeta_{t+1}}= e^{\psi_0S_t}.
\]
For any finite value $M>0$, let $\tau$ be the first time when the random walk leaves $(-M, 0]$, i.e.
\[
\tau = \min \left\{t\geq 1 :\, S_\tau > 0  \;\text{or}\; S_\tau \leq -M\right\}
\]
Then $\tau$ is a stopping time, which is almost surely finite since $\EE\,\zeta_1 < 0$ implies $S_t \to -\infty$ almost surely. 
Since $|S_{t\wedge\tau}| \leq M+1$ for all $t\geq 1$, we have by the optional stopping theorem
\begin{align}
1 &\;=\; \EE e^{\psi_0 S_\tau} \\
\label{eq:EStau}
&\;=\; p_0(M) \EE [e^{\psi_0 S\tau}|S_\tau > 0] + (1-p_0(M)) \EE [e^{\psi_0 S_\tau}|S_\tau \leq -M],
\end{align}
where $p_0(M)$ is the probability that $S_t$ reaches $(0, \infty)$ before it reaches $(-\infty, -M)$.
Since $S_t$ can increase no more than 1 at a time, we have $S_\tau\leq 1$. Therefore
\[
\EE [e^{\psi_0 S_\tau}|S_\tau > 0] \leq e^{\psi_0},
\]
and on the other hand, we have
\[
\EE [e^{\psi_0 S_\tau}|S_\tau \leq -M] \leq e^{-\psi_0 M}.
\]
Substituting both bounds into \eqref{eq:EStau} and isolating $p_0(M)$, we have
\[
p_0(M) \;\geq\; \frac{1 - e^{-\psi_0 M}}{e^{\psi_0} - e^{-\psi_0 M}}.
\]
Taking $M \to \infty$ we obtain
\[
\PP\left( \max_{t \geq 1} S_t > 0\right) \;\geq\; \sup_{M > 0} p_0(M) = e^{-\psi_0}.
\]
We pause to show the existence of a unique positive solution $\psi_0$ to equation \eqref{eq:MGFmartingale}. Defining $\lambda = e^{-\psi_0}$, we are equivalently seeking a unique solution in $(0,1)$ to the polynomial equation
\begin{equation}
\label{eq:strongmarkovian}
 \lambda = 1 - q + q\lambda^{1 + \alpha^{-1}} = 1 - q + q\lambda^{1/p}.
\end{equation}
Let $g(\lambda) = q\lambda^{1/p} - \lambda + 1-q$. Then $g(0) = 1-q > 0$ and $g(1) = 0$. Furthermore, 
\[
g'(\lambda) = \frac{q}{p}\lambda^{-1+1/p} - 1 = \frac{q}{p} \lambda^{1/\alpha} - 1.
\]
Hence, $g$ has a unique minimum at $\lambda_* = (p/q)^\alpha \in (0,1)$, with $g'(\lambda) < 0$ for $\lambda \in (0,\lambda_*)$ and $g'(\lambda) > 0$ for $\lambda \in (\lambda_*, 1)$. As a result $g(\lambda) = 0$ has a unique solution in $(0,1)$. To complete the proof, we require an upper bound for $1-\lambda$, since we have already shown $\PP\left( \max_{t \leq 1} S_t \geq 0\right) \;\leq\; 1-\lambda$.

Let $1 - \lambda = q\theta$. Then \eqref{eq:strongmarkovian} can be expressed as
\begin{equation}
\label{eq:strongmarkovian_delta}
    1 - \theta = (1 - q\theta)^{1/p}.
\end{equation}
Taking the log on both sides and multiplying by $q$, we have
\[
\frac{q}{p} = \frac{q\log(1-\theta)}{\log(1-q\theta)}.
\]
The right hand side approaches 1 as $\theta \to 0$ and diverges to infinity as $\theta \to 1$, giving the trivial bound $\theta < 1$. We are mainly interested in improving this bound when $q/p$ is close to 1, in which case $\theta$ should be small.
Lemma \ref{lemma:rwdrifttechnical1} gives
\[
\frac{q}{p} - 1 \;\geq\; \frac{(1-q)}{2}\theta \quad\Rightarrow\quad \theta \;\leq\; \frac{2(q-p)}{p(1-q)},
\]
so that, as desired,
\[
\PP\left(\max_{t \geq 1} S_t \leq 0\right) \;\leq\; 1 - \lambda \;\leq\; q \min\left\{\frac{2(q-p)}{p(1-q)}, \, 1\right\}.
\]
\end{proof}

\begin{lemma}
\label{lemma:pr}
Consider the following random walk defined in the proof of Proposition \ref{prop:bigokd}:
\[
\tilde{S}_k = \sum_{j=1}^k (p - \tilde{Z}_j), \quad \tilde{Z}_j \simiid \text{Bern}\left(q_{\delta}\right), \quad
\]
where $q_{\delta} > p$.
Define
\[
\lambda^* = p + \frac{q_{\delta} - p}{4} < q_\delta, \quad \text{and} \quad c_h= - \left[\lambda^*\log \frac{q_\delta}{\lambda^*} + (1 - \lambda^*)\log \frac{1 - q_\delta}{1 - \lambda^*}\right].
\]
Then for any
\[
r > 4k^*\left(\frac{q_{\delta}}{p}-1\right)^{-1},
\]
we have
\[
p(r) \triangleq \PP\left(\max\left\{k: \tilde{S}_k \geq -pk^*\right\} \;=\; r\right) 
\leq \frac{q_\delta e^{1-c_hr}}{2\pi \sqrt{rp(1-p)}}  \,\cdot\, \min\left\{\frac{2(q_{\delta}-p)}{p(1-q_{\delta})}, \;1\right\}.
 \]
\end{lemma}
\begin{proof}
Define events
\[
A_1 = \left\{\sum_{j=1}^r \tilde{Z}_j = \lfloor pr+pk^* \rfloor\right\},
\]
and
\[
A_2 = \left\{\max_{k\geq 1} \sum_{j=r+1}^{r+k} (p - \tilde{Z}_j)  \leq 0 \right\}.
\]
We will show that 
\[
\left\{\max\left\{k: \tilde{S}_k \geq -pk^*\right\} \;=\; r\right\} \subset A_1\cap A_2.
\]
First, note that
$\max\left\{k: \tilde{S}_k \geq -pk^*\right\} = r$ implies the following three conditions:
\[
(1) \tilde{S}_r \geq -pk^*, \quad (2)  \tilde{S}_{r+1} \leq -pk^*, \quad \text{and} \quad (3) \max_{k\geq 1} (\tilde{S}_{r+k} - \tilde{S}_r) \leq 0.
\]
Recalling the definition of $\tilde{S}_r$, conditions (1) and (2) are equivalent to
\[
pr+pk^* + p - 1 \;\leq\; \sum_{j=1}^r \tilde{Z}_j \;\leq\; pr+pk^*.
\]
Since $\sum_{j=1}^r \tilde{Z}_j$ is an integer, one of the following two events must happen: (a) there exists no integer between $pr+pk^*$ and $pr+pk^*+p-1$. In this case, $\left\{\max\left\{k: \tilde{S}_k \geq -pk^*\right\} \;=\; r\right\}$ is an empty set; (b) there exists exactly one integer between $pr+pk^*$ and $pr+pk^*+p-1$. Then it follows that $\sum_{j=1}^r \tilde{Z}_j = \lfloor pr+pk^* \rfloor$. In either case, we have shown that 
\[
\left\{\max\left\{k: \tilde{S}_k \geq -pk^*\right\} \;=\; r\right\} \subset A_1.
\]
In addition, condition (3) is equivalent to event $A_2$. Therefore, we have shown that
\[
\left\{\max\left\{k: \tilde{S}_k \geq -pk^*\right\} \;=\; r\right\} \subset A_1\cap A_2.
\]
Note that $A_1$ only depends on $\tilde{Z}_j, j \leq r$ and $A_2$ only depends on $\tilde{Z}_j, j > r$. Since $\{\tilde{Z}_j\}$ is a sequence of i.i.d Bernoulli variables, $A_1$ and $A_2$ are independent. Therefore
\[
p(r) \leq P(A_1)P(A_2).
\]
We now bound $P(A_1)$ and $P(A_2)$ separately.
First,
\[
P(A_1) = \text{Binom}(r, q_{\delta}; \lfloor pr+pk^* \rfloor).
\]
where $\text{Binom}(r, q_{\delta}; \lfloor pr+pk^* \rfloor)$ is the probability of the binomial distribution $\text{Binom}(r, q_{\delta}; \cdot)$ at $\lfloor pr+pk^* \rfloor$.
Let $m = \lfloor pr+pk^* \rfloor$, we have
\[
\text{Binom}(r, q_{\delta}; \lfloor pr+pk^* \rfloor) = \frac{r!}{m!(r-m)!}e^{m\log q_{\delta}}e^{(r-m)\log(1 - q_{\delta})}.
\]
Using Sterling's lemma, we obtain
\[
\frac{r!}{m!(r-m)!} \leq \frac{e}{2\pi}\sqrt{\frac{r}{m(r-m)}}e^{r\log r - m\log m - (r-m)\log (r-m)}
\]
Since $m \geq pr$, we know that 
\[
\sqrt{\frac{r}{m(r-m)}} \leq \sqrt{\frac{1}{rp(1-p)}}
\]
Therefore
\[
\text{Binom}(r, q_{\delta}; m) \leq
\frac{e}{2\pi}\sqrt{\frac{r}{m(r-m)}} \exp\left(m\log\frac{rq_{\delta}}{m} + (r-m)\log\frac{r(1-q_{\delta})}{r-m}\right).
\]
Let $\lambda = m/r \leq p + (q_\delta-p)/4 <q_\delta$. Then we have
\[
m\log\frac{rq_{\delta}}{m} + (r-m)\log\frac{r(1-q_{\delta})}{r-m} = r\left[\lambda\log \frac{q_\delta}{\lambda} + (1 - \lambda)\log \frac{1 - q_\delta}{1 - \lambda}\right].
\]
Note that the function
\[
g(\lambda) = \lambda\log \frac{q}{\lambda} + (1 - \lambda)\log \frac{1 - q_\delta}{1 - \lambda }
\]
is the negative of the KL divergence between $\text{Bernoulli}(\lambda)$ and $\text{Bernoulli}(q_\delta)$, and is thus 
increasing from $\lambda \in (0, q_\delta)$, and decreasing from $(q_\delta, 1)$. As such, we have
\[
m\log\frac{rq_{\delta}}{m} + (r-m)\log\frac{r(1-q_{\delta})}{r-m} \leq r \cdot g(p + (q_\delta-p)/4).
\]
Therefore
\[
P(A_1) \leq \frac{e}{2\pi \sqrt{rp(1-p)}}e^{-c_hr}.
\]
Now we bound the probability of event $A_2$. Using Lemma $\ref{lemma:rwdrift}$ with $N = 1/\alpha$, we get
\[
P(A_2) \leq q_\delta \min\left\{\frac{2(q_{\delta}-p)}{p(1-q_{\delta})}, \;1\right\}.
\]
Therefore
\[
p(r) \leq P(A_1)P(A_2) \leq \frac{q_\delta e^{1-c_hr}}{2\pi \sqrt{rp(1-p)}}  \,\cdot\, \min\left\{\frac{2(q_{\delta}-p)}{p(1-q_{\delta})}, \;1\right\}.
\]
The proof is now complete.
\end{proof}

\section*{Acknowledgments}
William Fithian is partially supported by the NSF DMS-1916220 and a Hellman Fellowship
from Berkeley. We are grateful to Stefan Wager, Nike Sun, Rina Barber, and Bin Yu for insights we gained from discussions with them, and especially Lihua Lei whose deep insights into knockoffs communicated over many conversations have greatly shaped our way of thinking about the method. In addition we would like to thank Tijana Zrnic, Art Owen, Dan Kluger, Lucas Janson, and Asher Spector for their thoughtful feedback on a previous draft.

\bibliographystyle{apalike}
\bibliography{knockoff.bib}

\end{document}